\documentclass[11pt]{article}

\usepackage{graphicx}
\usepackage[utf8]{inputenc}
\usepackage[T1]{fontenc}
\usepackage{amsmath}
\usepackage{amsthm}
\usepackage{amsfonts}
\usepackage{amssymb}
\usepackage[shortlabels]{enumitem}

\usepackage[normalem]{ulem}

\usepackage{natbib}

\usepackage{url}
\usepackage{slashed}
\usepackage{bbm}
\usepackage{tikz}
\usepackage{tikz-cd}
\usepackage{mathtools}
\usepackage{esint}

\usepackage{ifthen}
\usepackage{xspace}
\usepackage{fancyhdr}
\usepackage{aliascnt}
\usepackage[textsize=small]{todonotes}
\usepackage{bookmark}
\usepackage{caption}

\usetikzlibrary{calc}
\usepackage{cancel}
\usepackage{array}

\usepackage{array,booktabs,xcolor}
\definecolor{pos}{RGB}{200,0,0}      
\definecolor{spos}{RGB}{0,0,200}     
\definecolor{zcol}{gray}{0.35}

\usepackage[margin=1in]{geometry}

\usepackage[T1]{fontenc}

\usepackage{mlmodern}
\usepackage{eucal}

\usepackage{microtype}

\usepackage{setspace}
\tikzset{curve/.style={settings={#1},to path={(\tikztostart)
    .. controls ($(\tikztostart)!\pv{pos}!(\tikztotarget)!\pv{height}!270:(\tikztotarget)$)
    and ($(\tikztostart)!1-\pv{pos}!(\tikztotarget)!\pv{height}!270:(\tikztotarget)$)
    .. (\tikztotarget)\tikztonodes}},
    settings/.code={\tikzset{quiver/.cd,#1}
        \def\pv##1{\pgfkeysvalueof{/tikz/quiver/##1}}},
    quiver/.cd,pos/.initial=0.35,height/.initial=0}

\newtheorem{theorem}{Theorem}[section]
\newtheorem{prop}[theorem]{Proposition}
\newtheorem{lemma}[theorem]{Lemma}
\newtheorem{cor}[theorem]{Corollary}

\theoremstyle{definition}
\newtheorem{definition}[theorem]{Definition}

\newtheorem{rmk}[theorem]{Remark}

\newcommand{\R}{\mathbb{R}}
\newcommand{\C}{\mathbb{C}}

\newcommand{\SO}{\mathrm{SO}}

\newcommand{\U}{\mathrm{U}}

\newcommand{\del}{\partial}
\newcommand{\delbar}{\overline{\partial}}

\newcommand{\llangle}{\langle\!\langle}
\newcommand{\rrangle}{\rangle\!\rangle}

\usepackage{hyperref}
\hypersetup{
	colorlinks,
	linkcolor={red!50!black},
	citecolor={blue!50!black},
	urlcolor={blue!80!black}
}
\numberwithin{equation}{section}

\begin{document}
	
\title{The second variation of $2$-spheres in the $2n$-sphere}
	
\author{Gavin Ball, Jesse Madnick}
\date{September 2026}

\newcommand{\Addresses}
{{  \bigskip
\noindent	\textsc{University of Missouri} \par\nopagebreak
\noindent	\textsc{Columbia, MO, United States} \par\nopagebreak
\noindent	\texttt{gavin.ball@missouri.edu} \\

\medskip
\noindent	\textsc{Seton Hall University} \par\nopagebreak
\noindent	\textsc{South Orange, NJ, United States} \par\nopagebreak
\noindent	\texttt{jesse.ochs.madnick@gmail.com} \\
}}
	
\maketitle
	
\begin{abstract}
	We prove sharp upper and lower bounds for the Morse index and nullity of linearly full branched minimal $2$-spheres in the round $2n$-sphere. In particular, we show that the Morse index of a linearly full minimal $2$-sphere in the round $2n$-sphere is at least $n(n-1)(2n+1),$ with equality if and only if the $2$-sphere is twistor-equivalent to the Bor\r{u}vka sphere. Our techniques also apply more generally to give bounds for the Morse index and nullity for totally isotropic surfaces in the $2n$-sphere.
\end{abstract}

{\small
\tableofcontents}
	
\section{Introduction}

\indent \indent The study of minimal 2-spheres in the unit sphere is a classical topic in differential geometry. The foundations  were laid by Calabi \cite{Calabi67}, who uncovered a beautiful connection with the geometry of holomorphic curves. Calabi's work, developed further by Chern \cite{Chern70} and Barbosa \cite{Barbosa75}, shows that a minimal $2$-sphere can be linearly full only in an even-dimensional sphere and that the area of such linearly full minimal 2-spheres $S^2 \to S^{2n}$ is quantized and bounded below by $2 \pi n(n+1).$ Here, a map is said to be \emph{linearly full} if its image is contained in no proper totally geodesic subsphere.

In this article, we study the second variation of area of linearly full (branched) minimal immersions $f: S^2 \to S^{2n}$. Our main objects of interest are the Morse index and nullity. The Morse index, $\operatorname{Ind}(f),$ counts the number of variations that decrease area to second-order, while the nullity, $\operatorname{Null}(f),$ counts the number of Jacobi fields. These two quantities are fundamental invariants of minimal submanifolds and are the subject of several deep and influential conjectures.

We obtain two-sided bounds for both the index and the nullity. For the index, we have:
\begin{theorem}\label{thm:introS2IndexTheorem}
    Let $f: S^{2} \to S^{2n}$ be a linearly full, branched minimal immersion. Then
    \begin{equation*}
        (n-2) \frac{\operatorname{Area}(f)}{\pi} + n(n+3) \leq \operatorname{Ind}(f) \leq (n-1) \frac{\operatorname{Area}(f)}{\pi} - n(n-1).
    \end{equation*}
    In particular,
    \begin{equation*}
        \operatorname{Ind}(f) \geq n(n-1)(2n+1),
    \end{equation*}
    with equality if and only if $f$ is twistor-equivalent to the Bor\r{u}vka immersion.
\end{theorem}
Before stating the nullity bounds, we introduce some terminology. By the work of Calabi \cite{Calabi67} and Chern \cite{Chern70}, the osculating spaces of a linearly full branched minimal immersion $f : S^2 \to S^{2n}$ define a map $\widehat{\Psi} : S^2 \to \mathrm{Fl}$ into the isotropic flag manifold of $\C^{2n+1},$
\begin{equation*}
    \mathrm{Fl} = \lbrace 0 = P_{n+1} \subset P_n \subset \cdots \subset P_1 \subset \C^{2n+1} \mid \dim_\C P_r = n-r+1, \:\: P_1 \: \text{isotropic} \rbrace.
\end{equation*}
Projecting to individual constituents of the flag gives a sequence of holomorphic maps $\Psi_j,$ $1 \leq j \leq n,$ each mapping from $S^2$ into the Grassmannian of isotropic $(n-j+1)$-planes in $\C^{2n+1}.$ Let $b_j \geq 0$ denote the total ramification degree of $\Psi_j.$ In particular, $b_1$ gives the total branching order of $f.$ Define the \emph{intermediate ramification} $\mathrm{R}(f) \geq 0$ by
\begin{equation*}
    \mathrm{R}(f) = \sum_{j=1}^{n-1} (n-j)b_j.
\end{equation*}

\begin{theorem}\label{thm:introS2NullityTheorem}
    Let $f: S^{2} \to S^{2n}$ be a linearly full, branched minimal immersion. Then
    \begin{equation*}
        \frac{\operatorname{Area}(f)}{\pi}+2(n^2 - 3) - 2 b_1 \leq \operatorname{Null}(f) \leq \frac{\operatorname{Area}(f)}{\pi}+2(n^2 - 3) - 2 b_1 + 4 \mathrm{R}(f).
    \end{equation*}
    In particular,
    \begin{equation*}
        \operatorname{Null}(f) \geq (2n+3)(2n-2),
    \end{equation*}
    with equality if and only if $f$ is twistor-equivalent to the Bor\r{u}vka immersion.
\end{theorem}

The arguments used to prove Theorems \ref{thm:introS2IndexTheorem} and \ref{thm:introS2NullityTheorem} apply more generally to \emph{totally isotropic} immersions $f: \Sigma \to S^{2n}$ where $\Sigma$ is a compact Riemann surface (see \S\ref{sect:totallyisotropic} for the definition). A theorem of Calabi \cite{Calabi67} implies that any branched minimal immersion $f : S^2 \to S^{2n}$ is automatically totally isotropic. For genus $g \geq 1,$ not every branched minimal immersion is totally isotropic, but the totally isotropic examples still constitute a large class: a result of Bryant \cite{Bry82} in the case $n=2,$ generalized by Hano \cite{Hano96Immersions} to all $n \geq 2,$ implies that every compact Riemann surface $\Sigma$ admits a totally isotropic immersion $f : \Sigma \to S^{2n}.$

As in the case of the $2$-sphere, the osculating spaces of a totally isotropic branched immersion $f$ define a map $\widehat{\Psi} : \Sigma \to \mathrm{Fl},$ giving a sequence of holomorphic maps $\Psi_j,$ $1 \leq j \leq n,$ from $\Sigma$ into the Grassmannian of isotropic $(n-j+1)$-planes in $\C^{2n+1}.$ Pulling back the tautological flag over $\mathrm{Fl}$ gives rise to a holomorphic flag of bundles $\mathcal{P}_{j}$ on $\Sigma.$ Let $\mathcal{G}_j = (\mathcal{P}_j / \mathcal{P}_{j+1})^*,$ a holomorphic line bundle over $\Sigma$ having $\deg(\mathcal{G}_j) \geq 2j(1-g)$, and let $\mathcal{K}$ denote the canonical bundle of $\Sigma.$

\newpage

\begin{theorem}\label{thm:introgeneralindex}
    Let $f: \Sigma \to S^{2n}$ be a compact, linearly full, totally isotropic branched immersion, where $\Sigma$ has genus $g.$ Then
	\begin{equation*}
		\operatorname{Ind} (f) \geq (n-2) \frac{\operatorname{Area}(f)}{\pi} + 2 \deg \mathcal{G}_n + n(n-1)(1-g) - 2 \sum_{k=1}^{n-2} h^1(\mathcal{G}_k \otimes \mathcal{G}_{k+1}),
	\end{equation*}
    and
    \begin{equation*}
        \operatorname{Ind}(f) \leq (n-1) \left( \frac{\operatorname{Area}(f)}{\pi} - n(1-g) \right) + 2 \sum_{k=1}^{n-1} h^1(\mathcal{K} \otimes \mathcal{G}_k \otimes \mathcal{P}_{k+1}^*).
    \end{equation*}
\end{theorem}

\begin{theorem}\label{thm:introgeneralnull}
    Let $f: \Sigma \to S^{2n}$ be a compact, linearly full, totally isotropic branched immersion, where $\Sigma$ has genus $g.$ Set
    \begin{equation*}
        \operatorname{N}(f) = \frac{\operatorname{Area}(f)}{\pi} + 2(n^2-3)(1-g) -2b_1 + 2 h^1(\mathcal{G}_{n-1} \otimes \mathcal{G}_n).
    \end{equation*}
    Then
    \begin{equation*}
        \operatorname{N}(f) \leq \operatorname{Null}(f) \leq \operatorname{N}(f) + 4 \operatorname{R}(f) + 4 \sum_{k=1}^{n-2} h^1(\mathcal{G}_k \otimes \mathcal{G}_{k+1}).
    \end{equation*}
\end{theorem}

For $\Sigma \cong S^2,$ all the $h^1$ terms vanish for degree reasons. Together with classical degree identities, this gives Theorems \ref{thm:introS2IndexTheorem} and \ref{thm:introS2NullityTheorem} as specializations of Theorems \ref{thm:introgeneralindex} and \ref{thm:introgeneralnull}. In fact, for $\Sigma \cong S^2,$ when the first $n-1$ holomorphic curves $\Psi_1, \ldots, \Psi_{n-1}$ are unbranched, so that $\operatorname{R}(f) = 0,$ our bounds produce exact equalities for both the index and nullity:
\begin{equation*}
    \operatorname{Ind}(f) = (n-1)\left(\frac{\operatorname{Area}(f)}{\pi}-n\right), \quad \operatorname{Null}(f) = \frac{\operatorname{Area}(f)}{\pi}+2(n^2-3).
\end{equation*}

In the case where $\Sigma$ has genus $g=1,$ cohomology calculations and degree bounds give, for $n \geq 3,$
\begin{equation*}
\begin{aligned}
    & 2(n-1)(n^2+1) + 2 \left\lfloor \frac{n}{2} \right\rfloor \leq \operatorname{Ind}(f) \leq (n-1)\frac{\operatorname{Area}(f)}{\pi},
\end{aligned} 
\end{equation*}
as well as a refined nullity bound (see \S\ref{ssect:genusone}).

The idea behind the proofs of Theorems \ref{thm:introgeneralindex} and \ref{thm:introgeneralnull} is to realize the (complexified) second variation of $f$ as the first term $q_0$ in a sequence of quadratic forms $q_0, \ldots, q_{n-1}$ defined on the spaces of sections of a sequence of holomorphic bundles $\mathcal{E}_0, \ldots, \mathcal{E}_{n-1}$ over $\Sigma.$ This sequence of quadratic forms has the property that its final member $q_{n-1}$ is non-negative and that consecutive forms $q_k$ and $q_{k+1}$ in the sequence satisfy an identity
\begin{equation*}
    q_k(W)-q_{k+1}(Z) =	\left\lVert	\overline{\partial}_{\mathcal E_k}W-P^\dagger_{k+1}Z \right\rVert^2 - \left\lVert \partial_{\mathcal E_{k+1}}Z+P_{k+1}W \right\rVert^2,
\end{equation*}
where the $P_{k}$ are certain zeroth-order operators derived from the geometry of $\Sigma.$ This identity allows one to move negative-definite subspaces up and down the sequence of bundles $\mathcal{E}_j,$ provided certain $\del$- or $\delbar$-equations can be satisfied: solving a $\delbar$-equation removes the positive term and transfers negative directions from $q_{k+1}$ to $q_k,$ giving a lower bound on the index; while solving a $\del$-equation does the reverse, giving an upper bound on the index. The obstructions to solving these equations are naturally cohomology spaces. Similar considerations allow us to give upper and lower bounds on the nullity.

\subsection{Context}

\indent \indent The study of the second variation of minimal submanifolds has a long history. In the fundamental work of Simons \cite{Simons68}, ambient vector fields from $\R^{N+1}$ are used to prove instability of minimal submanifolds of $S^N.$ For minimal $2$-spheres in $S^N,$ Ejiri \cite{Ejiri83} identifies the lowest eigenspace of the Jacobi operator with the space of holomorphic sections of the normal bundle, proving the lower bound
\begin{equation*}
    \operatorname{Ind} (f) \geq \frac{\operatorname{Area}(f)}{\pi} + 2(n-3) 
\end{equation*}
in the linearly full unbranched case. The argument for our lower bound begins with this observation and continues along the sequence $q_k$ to construct additional negative directions.

In the case $n=2,$ Montiel--Urbano \cite{MontUrb97} give an exact formula for the index of a compact totally isotropic immersion $f: \Sigma \to S^4,$ namely
\begin{equation*}
    \operatorname{Ind} (f) = \frac{\operatorname{Area}(f)}{\pi} + 2g-2.
\end{equation*}
They also obtain bounds on the nullity. Our general bounds coincide when $n=2$ and $f$ is unbranched, and indeed our work is inspired by the arguments in \cite{MontUrb97}.

In the case $n=3,$ the second author \cite{Madnick22} studies null-torsion holomorphic curves in the nearly K\"ahler $S^6,$ computing the multiplicity of the lowest Jacobi eigenvalue under a genus--area hypothesis and obtaining a lower bound for the nullity in all genera.

Karpukhin \cite{Karpukhin21} gives a lower bound for the index of a linearly full branched minimal $2$-sphere $f: S^2 \to S^{2n},$
\begin{equation*}
    \operatorname{Ind}(f) \geq (n-1)\left( \frac{\operatorname{Area}(f)}{\pi} + 4 - 2 \left\lfloor \sqrt{2\frac{\operatorname{Area}(f)}{\pi} + 1} \right\rfloor_{\mathrm{odd}}  \right),
\end{equation*}
where $\lfloor x \rfloor_{\mathrm{odd}}$ denotes the largest odd integer not exceeding $x.$ For fixed $n,$ Karpukhin's lower bound is more restrictive than the lower bound of Theorem \ref{thm:introS2IndexTheorem} for sufficiently large values of the area, while the lower bound of Theorem \ref{thm:introS2IndexTheorem} is sharp at the minimum area and our ramification-sensitive bounds can give an exact index when there is no intermediate ramification. Combining Karpukhin's lower bound and the upper bound of Theorem \ref{thm:introS2IndexTheorem} gives a good understanding of the asymptotics of the index along a sequence of branched totally isotropic spheres with increasing area. Karpukhin also proves that the index of a linearly full minimal immersion $f: \mathbb{RP}^2 \to S^{2n}$ equals half that of the corresponding lift $\widehat{f} : S^2 \to S^{2n}.$ This observation can be combined with our Theorem \ref{thm:introS2IndexTheorem} to obtain new bounds on the index for linearly full minimal immersions $\mathbb{RP}^2 \to S^{2n}.$

Kusner--Wang \cite{KusWang24} give index lower bounds for minimal $2$-tori in round spheres. Their results include minimal tori in $S^4$ with non-vanishing Hopf differential and minimal tori in higher codimension with vanishing Hopf differential. The latter condition is weaker than total isotropy. Their estimates therefore concern a wider class of tori compared to those under consideration here. The more restrictive hypotheses of this article allow us to make use of holomorphic methods to prove stronger bounds. 

The equality statement of Theorem \ref{thm:introS2IndexTheorem} answers in the affirmative a question raised by the authors in \cite{BallMadnick26}. In addition, the results of this paper can be used to give alternative derivations of the formulas in \cite{BallMadnick26} for the index and nullity of the Bor\r{u}vka immersions.

Our nullity bounds can be compared with the moduli space theory for harmonic maps $\Sigma \to S^{2n}.$ Recall that for two-dimensional domains branched minimal immersions are equivalent to weakly conformal harmonic maps. Fern\'andez \cite{Fernandez12} shows that the space of linearly full harmonic maps $S^2 \to S^{2n}$ with fixed \emph{harmonic degree} $d = \operatorname{Area}(f)/4\pi$ has pure complex dimension $2d+n^2.$ For unbranched immersions, subtracting the six real dimensions of conformal reparametrizations gives a space of deformations of real dimension $4d+2n^2-6,$ which agrees with the lower bound from Theorem \ref{thm:introS2NullityTheorem}. Lemaire--Wood \cite{LemaireWood09} study the space of Jacobi fields along harmonic 2-spheres in $S^3$ and $S^4,$ proving by means of twistor theory that  Jacobi fields are not always integrable in this setting. 

In a very recent preprint, Caniato \cite{Caniato26} constructs families of linearly full, totally isotropic branched immersions of any fixed compact Riemann surface into $S^{2n}$, of complex dimension $2d+n^2(1-g)$, for an unbounded sequence of harmonic degrees $d=\operatorname{Area}(f)/4\pi$. These families arise near certain branched covers of the Bor\r{u}vka immersion, and the normal components of the infinitesimal deformations through these families span a space of real dimension $4d+2(n^2-3)(1-g)-2b_1,$ which agrees with our lower nullity bound of Theorem \ref{thm:introgeneralnull} applied to these examples. Thus, in Caniato's examples the dimension in our lower nullity bound is realized by integrable Jacobi fields. Our lower bound also provides evidence for Conjecture 1.2 of \cite{Caniato26}, while our upper bound limits the possible additional nullity.

\subsection{Organization}

\indent \indent The first three sections after the introduction of this article (\S\ref{sect:thesphere}-\S\ref{sect:secondvar}) are primarily expository. In \S\ref{sect:thesphere}, we set up the moving frame on $S^{2n}$ and describe the geometry of some related homogeneous spaces: the partial twistor and flag manifolds. In \S\ref{sect:totallyisotropic}, we define total isotropy for branched immersions $f: \Sigma \to S^{2n}$ and set up their structure equations. We also set up the holomorphic apparatus used to prove our results: we define certain holomorphic bundles over the domain $\Sigma$ of a totally isotropic branched immersion and study their degree and curvature. In \S\ref{sect:secondvar}, we look at the second variation formula in our context and reprove a result of Ejiri as motivation for our later calculation.

The heart of the paper is \S\ref{sect:bounds}. Here, we define the sequence $q_k$ of quadratic forms and prove the fundamental identity (\ref{eq:qkdiffid}) relating $q_k$ and $q_{k+1}$ via an integral Weitzenb\"ock formula. Remark \ref{rmk:twistorinterpret} gives a twistor interpretation of this identity. Then, in the remaining part of \S\ref{sect:bounds}, we use (\ref{eq:qkdiffid}) to prove our general index and nullity bounds. In \S\ref{sect:lowgenus}, we apply the general results to the low genus cases $g = 0$ and $g=1.$

\subsection*{AI use disclosure}

The authors used AI (LLMs) in the preparation of this article, primarily to proofread drafts and as a component of the process of checking arguments for correctness. The words in this article are our own. The strategy used to prove the main results is due to the authors. The only noteworthy mathematical contribution of AI is the suggestion, in the course of a ``chat'' with the first author, that in the case of the Bor\r{u}vka immersions $f : S^2 \to S^{2n}$ the negative variations constructed in \cite{BallMadnick26} correspond to sections of the bundles called $\mathcal{E}_k$ in \S\ref{sect:bounds}.
	
\section{The sphere $S^{2n}$}\label{sect:thesphere}
	
\subsection{The moving frame}
	
\indent \indent Let $\R^{2n+1}$ be Euclidean space with inner product $\langle \cdot, \cdot \rangle.$ We regard the unit sphere $S^{2n}$ as the homogeneous space
\begin{equation*}
	S^{2n} = \SO(2n+1) / \SO(2n).
\end{equation*}
Writing an element of $\SO(2n+1)$ as an oriented orthonormal frame $g = (\mathbf{x}, \mathbf{e}_1, \ldots, \mathbf{e}_{2n}),$ the quotient map $\SO(2n+1) \to S^{2n}$ is projection $g \mapsto \mathbf{x}.$ Let
\begin{equation*}
	\omega = g^{-1} d g \in \Omega^1(\SO(2n+1);\mathfrak{so}(2n+1))
\end{equation*}
be the left-invariant Maurer-Cartan form of $\SO(2n+1)$, and set $\theta_{\alpha} = \omega_{\alpha 0}$ for $1 \leq \alpha \leq 2n,$ so that $dg = g\omega$ gives
\begin{equation}\label{eq:frameeqsR}
	\begin{aligned}
		d \mathbf{x}          & = \theta_{\alpha} \mathbf{e}_{\alpha},                                   \\
		d \mathbf{e}_{\alpha} & = - \theta_{\alpha} \mathbf{x} + \omega_{\beta \alpha} \mathbf{e}_\beta. 
	\end{aligned}
\end{equation}
The Maurer-Cartan equation $d \omega + \omega \wedge \omega = 0$ becomes the structure equations
\begin{equation}\label{eq:structeqsR}
	\begin{aligned}
		d \theta_{\alpha}       & = -\omega_{\alpha \beta} \wedge \theta_{\beta},                                                  \\
		d \omega_{\alpha \beta} & = - \omega_{\alpha \gamma} \wedge \omega_{\gamma \beta} + \theta_{\alpha} \wedge \theta_{\beta}. 
	\end{aligned}
\end{equation}
	
We next rewrite these equations in complex notation. For an index $1 \leq a \leq n,$ set $p_a = 2a-1, q_a = 2a$ and define
\begin{equation*}
	\begin{aligned}
		\mathbf{f}_{\overline a} & = \mathbf{e}_{p_a} - i \mathbf{e}_{q_a}, \quad \mathbf{f}_a = \overline{\mathbf{f}_{\overline{a}}}, \\
		\eta_a                   & = \theta_{p_a} + i \theta_{q_a},                                                                    
	\end{aligned}
\end{equation*}
and
\begin{equation*}
	\begin{aligned}
		\kappa_{a \overline{b}} & = \tfrac{1}{2} ( \omega_{p_a p_b} + \omega_{q_a q_b} + i \omega_{q_a p_b} - i \omega_{p_a q_b}), \\
		\sigma_{ab}             & = \tfrac{1}{2} ( \omega_{p_a p_b} - \omega_{q_a q_b} + i \omega_{q_a p_b} + i \omega_{p_a q_b}). 
	\end{aligned}		
\end{equation*}
These forms satisfy $\kappa_{a \overline{b}} = - \overline{\kappa_{b \overline{a}}}$ and $\sigma_{ab} = -\sigma_{ba}.$ If $\langle \cdot, \cdot \rangle_\C$ denotes the $\C$-bilinear extension of $\langle \cdot, \cdot \rangle,$ then we have
\begin{equation*}
	\langle \mathbf{f}_a, \mathbf{f}_b \rangle_\C = 0, \quad \langle \mathbf{f}_a, \mathbf{f}_{\overline{b}} \rangle_\C = 2 \delta_{ab}.
\end{equation*}
The Lie algebra $\mathfrak{so}(2n+1)$ decomposes under the subgroup $\U(n)$ as $\C^n \oplus \mathfrak{u}(n) \oplus \Lambda^2_\C \C^n,$ and the splitting $\omega = \eta + \kappa + \sigma$ respects this decomposition. In this notation, (\ref{eq:frameeqsR}) becomes
\begin{equation}\label{eq:frameeqsC}
	\begin{aligned}
		d \mathbf{x}                & = \tfrac{1}{2} \mathbf{f}_{\overline{a}} \eta_a + \tfrac{1}{2} \mathbf{f}_a \overline{\eta_a},                              \\
		d \mathbf{f}_{\overline{a}} & = - \overline{\eta_a} \mathbf{x} + \mathbf{f}_{\overline{b}} \kappa_{b \overline{a}} - \mathbf{f}_b \overline{\sigma_{ab}}, \\
		d \mathbf{f}_a              & = - \eta_a \mathbf{x} - \mathbf{f}_b \kappa_{a \overline{b}} - \mathbf{f}_{\overline{b}} \sigma_{ab},                       
	\end{aligned}
\end{equation}
while the structure equations (\ref{eq:structeqsR}) become
\begin{equation}\label{eq:structeqsC}
	\begin{aligned}
		d \eta_a                  & = -\kappa_{a \overline{b}} \wedge \eta_b - \sigma_{ab} \wedge \overline{\eta_b},                                                                       \\
		d \kappa_{a \overline{b}} & = - \kappa_{a \overline{c}} \wedge \kappa_{c \overline{b}} - \sigma_{ac} \wedge \overline{\sigma_{cb}} + \tfrac{1}{2} \eta_a \wedge \overline{\eta_b}, \\
		d \sigma_{ab}             & = - \kappa_{a \overline{c}} \wedge \sigma_{cb} + \kappa_{b \overline{c}} \wedge \sigma_{ca} + \tfrac{1}{2} \eta_a \wedge \eta_b.                       
	\end{aligned}
\end{equation}
	
\subsection{Related homogeneous spaces}

In our work, we shall study the geometry of surfaces in $S^{2n}$ using the geometry of several homogeneous spaces of $\SO(2n+1).$ In this section we describe these spaces and their natural $\SO(2n+1)$-invariant geometric structures.

\subsubsection{Isotropic planes}

\indent \indent As above, let $\langle \cdot, \cdot \rangle_\C$ denote the $\C$-linear extension to $\C^{2n+1}$ of the standard inner product $\langle \cdot, \cdot \rangle$ on $\R^{2n+1}.$ The complex Lie group $\SO(2n+1, \C)$ acts on $\C^{2n+1}$ preserving $\langle \cdot, \cdot \rangle_\C.$ In this subsection, we fix an oriented orthonormal basis $\mathbf{x}, \mathbf{e}_1, \ldots, \mathbf{e}_{2n}$ and define the $\mathbf{f}_a$ as above.

\begin{definition}
	An $m$-dimensional complex subspace $S < \C^{2n+1}$ is called \emph{isotropic} if the restriction of $\langle \cdot, \cdot \rangle_\C$ to $S$ is identically zero. A $k$-dimensional complex subspace $U < \C^{2n+1}$ is called \emph{nondegenerate} if the restriction of $\langle \cdot, \cdot \rangle_\C$ to $U$ is nondegenerate.
\end{definition}

\indent The action of $\SO(2n+1, \C)$ on $\C^{2n+1}$ is transitive on isotropic $m$-planes, so every isotropic $m$-plane lies in the $\SO(2n+1, \C)$-orbit of the standard plane
\begin{equation*}
	S_m = \mathrm{span}_\C(\mathbf{f}_{n-m+1}, \ldots, \mathbf{f}_n).
\end{equation*}
The $\SO(2n+1, \C)$-stabilizer of $S_m$ is a parabolic subgroup $P_m < \SO(2n+1, \C).$ The action of $P_m$ on $\C^{2n+1}$ preserves the subspace
\begin{equation*}
	S_m^\perp = \left\lbrace u \in \C^{2n+1} \mid \langle u, S_m \rangle_\C = 0 \right\rbrace = \mathrm{span}_\C(\mathbf{x}, \mathbf{e}_1, \ldots, \mathbf{e}_{2n-2m}, \mathbf{f}_{n-m+1}, \ldots, \mathbf{f}_n)
\end{equation*}
and we have $S_m < S_m^{\perp}.$ The space $S_m^{\perp}/S_m$ inherits a nondegenerate bilinear form from $\langle \cdot, \cdot \rangle_\C.$ We have a chain of subspaces $S_m < S_m^{\perp} < \C^{2n+1}.$ In particular, there are $P_m$-equivariant short exact sequences of vector spaces
\begin{equation*}
	0 \to S_m \to S_m^\perp \to S_m^{\perp}/S_m \to 0,
\end{equation*}
and
\begin{equation*}
	0 \to S_m^{\perp} \to \C^{2n+1} \to S_m^* \to 0
\end{equation*}
where the map $\C^{2n+1} \to S_m^*$ in the second sequence sends $v \in \C^{2n+1}$ to the linear functional $s \mapsto \langle v, s \rangle_\C.$

The maximal compact subgroup $\SO(2n+1) < \SO(2n+1,\C)$ also acts transitively on the space of isotropic $m$-planes, with stabilizer $\SO(2n+1-2m) \times \U(m).$ Since $\SO(2n+1)$ preserves the standard real structure on $\C^{2n+1},$ this stabilizer group also preserves the subspace $\overline{S}_m,$ and hence the subspace $U_m = \left( S_m \oplus \overline{S_m} \right)^\perp = \mathrm{span}_\C(\mathbf{x}, \mathbf{e}_1, \ldots, \mathbf{e}_{2n-2m})$. This gives a $\SO(2n+1-2m) \times \U(m)$-invariant decomposition
\begin{equation}\label{eq:csplitcompact}
	\C^{2n+1} = \underbrace{S_m \oplus U_m}_{S_m^\perp} \oplus \overline{S_m},
\end{equation}
where $U_m$ is nondegenerate. We note that the decomposition (\ref{eq:csplitcompact}) is not invariant under the larger group $P_m.$
    
\subsubsection{The partial twistor spaces}
    
\begin{definition}
	For $1 \leq m \leq n$ the \emph{partial twistor space} $Z_m$ is defined to be the Grassmannian of isotropic $m$-planes in $\C^{2n+1}$,
	\begin{equation*}
		Z_m = \mathrm{IGr}(m, \C^{2n+1}) = \{E \in \mathrm{Gr}(m, \C^{2n+1}) \colon E \text{ is isotropic}\}.
	\end{equation*}
\end{definition}
As noted in the previous subsection, $Z_m$ carries transitive actions of the group $\SO_(2n+1, \C)$ and its maximal compact subgroup $\SO(2n+1),$ giving realizations $Z_m \cong \SO(2n+1, \C)/P_m \cong \SO(2n+1)/(\SO(2n+1-2m) \times \U(m)).$ In both cases the quotient map is given by
\begin{equation*}
	g \mapsto S_{m}(g) = \mathrm{span}_\C(\mathbf{f}_{n-m+1}, \ldots, \mathbf{f}_n).
\end{equation*}
The space $Z_m$ may be embedded in the projective space $\mathbb{P}(\Lambda^m \C^{2n+1})$ via the Pl\"ucker embedding. In terms of an adapted frame, this is the map
\begin{equation*}
    \mathrm{span}_\C(\mathbf{f}_{n-m+1}, \ldots, \mathbf{f}_n) \mapsto [\mathbf{f}_{n-m+1} \wedge \cdots \wedge \mathbf{f}_n].
\end{equation*}

We first describe the geometric structures on $Z_m$ invariant under the larger group $\SO_(2n+1,\C).$ The definition of $Z_m$ in terms of complex subspaces satisfying a set of algebraic equations naturally gives $Z_m$ the structure of a complex manifold. Let $\mathcal{S}_m \to Z_m$ be the tautological isotropic $m$-plane bundle over $Z_m.$ Repeating the linear algebra of the previous subsection gives short exact sequences of vector bundles over $Z_m,$
\begin{equation*}
	\begin{aligned}
		  & 0 \to \mathcal{S}_m \to \mathcal{S}_m^{\perp} \to \mathcal{S}_m^{\perp} / \mathcal{S}_m \to 0, \\
		  & 0 \to \mathcal{S}_m^{\perp} \to \underline{\C}^{2n+1} \to \mathcal{S}_m^* \to 0.               
	\end{aligned}
\end{equation*}
The tangent space to the ordinary Grassmannian $\mathrm{Gr}(m,\C^{2n+1})$ at an $m$-plane $E$ is naturally identified with $\mathrm{Hom}_\C(E, \C^{2n+1}/E).$ Differentiating the condition that $S \in Z_m$ be isotropic gives
\begin{equation*}
	T^{1,0}_S Z_m =\left\{A\in\operatorname{Hom}_\C(S,\C^{2n+1}/S) \mid \text{the map} \:\: (u,v) \mapsto \langle u, Av \rangle_\C \:\: \text{is skew-symmetric} \right\},
\end{equation*}
where the map $(u,v) \mapsto \langle u, Av \rangle_\C$ is well-defined by isotropy of $S.$ Therefore, there is an $\SO(2n+1,\C)$-equivariant holomorphic exact sequence
\begin{equation}\label{eq:tangentsesZm}
	0 \longrightarrow \mathcal{S}_m^*\otimes (\mathcal{S}_m^{\perp} / \mathcal{S}_m) \longrightarrow T^{1,0}Z_m \xrightarrow{\theta_m} \Lambda^2 \mathcal{S}_m^* \longrightarrow 0.
\end{equation}

\begin{definition}
	The \emph{horizontal distribution} on $Z_m$ is the holomorphic distribution
	\begin{equation*}
		\mathcal H_m^{1,0} = \ker\theta_m.
	\end{equation*}
	Thus $\mathcal H_m^{1,0} \cong \mathcal S_m^* \otimes (\mathcal{S}_m^{\perp} / \mathcal{S}_m)$ and $T^{1,0} Z_m /\mathcal{H}_m^{1,0} \cong \Lambda^2\mathcal S_m^*$. Write $\mathcal H_m\subset TZ_m$ for the corresponding real distribution.
\end{definition}

We now consider the geometric structures on $Z_m$ associated to the action of the compact group $\SO(2n+1)$. By the splitting (\ref{eq:csplitcompact}), an isotropic $m$-plane $S\subset\C^{2n+1}$ determines a real oriented $2m$-plane
\begin{equation*}
	V_S=(S \oplus \overline{S})\cap\R^{2n+1} = \mathrm{span}_\R(\mathbf{e}_{2n-2m + 1}, \ldots, \mathbf{e}_{2n})
\end{equation*}
together with an orientation-compatible orthogonal complex structure $J_S$ on $V_S$, characterized by the condition that $S$ be its $-i$-eigenspace. Conversely, such a pair $(V,J)$ determines an isotropic $m$-plane by taking the $-i$-eigenspace of $J$ on $V \otimes \C.$ Hence there is an $\SO(2n+1)$-equivariant fibration of $Z_m$ over the oriented Grassmannian
\begin{equation}\label{eq:partwist}
	\SO(2m)/\U(m) \longrightarrow Z_m \xrightarrow{\pi_m} \mathrm{Gr}^+(2m,\R^{2n+1}), \qquad \pi_m(S)=V_S,
\end{equation}
where the fibre over $V_S$ may be identified with the space of orientation-compatible orthogonal complex structures on $V_S.$ 
    
The tautological real $2m$-plane bundle over $\mathrm{Gr}^+(2m,\R^{2n+1})$ has a connection induced by the flat connection of the trivial bundle $\underline{\R}^{2n+1}$ via orthogonal projection. The resulting horizontal distribution on \eqref{eq:partwist} can be shown to equal $\mathcal{H}_m$ and we obtain an $\SO(2n+1)$-invariant smooth splitting $T^{1,0}Z_m = \mathcal{H}_m^{1,0} \oplus (\ker d\pi_m)^{1,0}$ satisfying
\begin{equation*}
	\mathcal{H}_m^{1,0} \cong \mathcal{S}_m^* \otimes (\mathcal{S}_m^{\perp} / \mathcal{S}_m), \qquad (\ker d\pi_m)^{1,0} \cong \Lambda^2\mathcal S_m^*.
\end{equation*}
In general, this smooth splitting is not a holomorphic splitting of \eqref{eq:tangentsesZm}.

This construction may be made more concrete using the structure equations. Let $r=n-m+1$ and fix indices $1 \leq i,j \leq r-1,$ $r \leq a, b, c \leq n.$ Then the forms $\kappa_{a \overline{b}},$ $\eta_i,$ $\kappa_{i \overline{j}},$ $\sigma_{ij}$ are connection forms for the projection $\SO(2n+1) \to Z_m$, while the forms $\eta_a,$ $\kappa_{a \overline{i}},$ $\sigma_{ai}$ and $\sigma_{ab}$ are semi-basic. The complex structure on $Z_m$ is equal to the one defined by declaring this latter set of forms to have type $(1,0)$. We have that $\mathcal{H}^{1,0}_m \subset T^{1,0} Z_m$ is given by
\begin{equation*}
	\mathcal{H}^{1,0}_m = \ker(\lbrace \sigma_{ab} \rbrace_{r \leq a < b \leq n}).
\end{equation*}
    
The two extreme values of $m$ are worth highlighting. When $m=1$, the fibre $\SO(2)/\U(1)$ is a point, so $Z_1 \subset \mathbb{CP}^{2n}$ is the null quadric $Q^{2n-1} \cong \mathrm{Gr}^+(2,\R^{2n+1})$.

When $m=n$, the space $Z_n = Z$ is known as the (full) \emph{twistor space} of $S^{2n}.$ The oriented Grassmannian $\mathrm{Gr}^+(2n,\R^{2n+1})$ is naturally identified with $S^{2n}$ by taking the orthogonal complement. Under this identification, (\ref{eq:partwist}) becomes the twistor fibration
\begin{equation*}
	\SO(2n)/\U(n) \to Z \to S^{2n}, \quad \pi(S_n(g)) = \mathbf{x}.
\end{equation*}

The full twistor space $Z$ admits an embedding $Z \to \mathbb{P}(\Delta_n),$ where $\Delta_n$ is the complex $2^n$-dimensional spinor representation of the group $\mathrm{Spin}(2n+1,\C)$, as we now explain. This embedding is due to Cartan \cite{Cartan81}. If $S \in Z$ is an $n$-dimensional isotropic plane, then the subspace $\ell_S < \Delta_n$ defined by
\begin{equation*}
    \ell_S = \left\lbrace \varphi\in\Delta_n \mid v \cdot \varphi = 0 \:\: \text{for every} \:\: v\in S \right\rbrace,
\end{equation*}
where $\cdot$ denotes Clifford multiplication, is a complex line. Cartan's embedding is the map
\begin{equation*}
    \mathrm{Ca} : Z \to \mathbb{P}(\Delta_n), \qquad S \mapsto \ell_S.
\end{equation*}
Therefore, $Z$ has two canonical projective embeddings, the Pl\"ucker embedding $\mathrm{Plu}: Z \to \mathbb{P}(\Lambda^n \C^{2n+1}),$ and Cartan's embedding $\mathrm{Ca} : Z \to \mathbb{P}(\Delta_n).$ The relation between these is
\begin{equation}\label{eq:spinorsquarerel}
    (\mathrm{Ca}^*\mathcal{O}_{\mathbb{P}(\Delta_n)}(1))^{\otimes 2} \cong \mathrm{Plu}^* \mathcal{O}_{\mathbb{P}(\Lambda^n \C^{2n+1})}(1),
\end{equation}
see \cite[\S109]{Cartan81} (also \cite{Hano96}). Here, we recall that for a complex vector space $V$, the notation $\mathcal{O}_{\mathbb{P}(V)}(1)$ refers to the hyperplane bundle over $\mathbb{P}(V)$.
    
\subsubsection{The flag manifold}

\indent \indent Let $T^n = \SO(2)^n < \SO(2n)$ be the diagonal maximal torus.

\begin{definition}
	The \emph{flag manifold} of $\SO(2n+1)$ is
	\begin{equation*}
		\mathrm{Fl} = \SO(2n+1)/T^n \cong \lbrace 0 = P_{n+1} \subset P_n \subset \cdots \subset P_1 \subset \C^{2n+1} \mid \dim P_r = n-r+1, \:\: P_1 \: \text{isotropic} \rbrace.
	\end{equation*}
\end{definition}
The quotient map $\SO(2n+1) \to \mathrm{Fl}$ sends $g$ to the flag with constituents $P_r(g) = \mathrm{span}_\C(\mathbf{f}_r, \ldots, \mathbf{f}_n).$ The forms $\kappa_{a \overline{a}}$ are connection forms for $\mathrm{Fl}$, while the forms $\eta_a,$ $\kappa_{a \overline{b}}, \sigma_{ab}$ with $a > b$ are semi-basic. There is an integrable complex structure on $\mathrm{Fl}$ defined by declaring the latter set of forms to have type $(1,0).$

For each $r$ there are holomorphic projections $\mathrm{Fl} \to Z_{n-r+1}$ given by discarding all members of the flag except $P_r.$ In the case of $Z_n = Z,$ composing with the twistor fibration gives a homogeneous fibration
\begin{equation*}
	\SO(2n)/T^n \to \mathrm{Fl} \to S^{2n}.
\end{equation*}
These maps fit into the diagram
\[\begin{tikzcd}
	&& {\mathrm{Fl}} & \\
	{Z_1} & {Z_2} & \cdots & {Z_n} \\
	{\mathrm{Gr}^+(2,\R^{2n+1})} & {\mathrm{Gr}^+(4,\R^{2n+1})} & \cdots & {\mathrm{Gr}^+(2n,\R^{2n+1}) \cong S^{2n}.}
	\arrow[from=1-3, to=2-1]
	\arrow[from=1-3, to=2-2]
	\arrow[from=1-3, to=2-3]
	\arrow[from=1-3, to=2-4]
	\arrow["{\pi_1}"', from=2-1, to=3-1]
	\arrow["{\pi_2}"', from=2-2, to=3-2]
	\arrow[from=2-3, to=3-3]
	\arrow["{\pi_n}", from=2-4, to=3-4]
\end{tikzcd}\]

Let $\mathcal{D}^{1,0}$ be the \emph{superhorizontal} distribution \cite{BursRawn90} defined by
\begin{equation}\label{eq:flagdistribution}
	\mathcal{D}^{1,0} = \ker\left( \lbrace \eta_a \rbrace_{a \geq 2} \cup \lbrace \kappa_{a \overline{b}} \rbrace_{a-b \geq 2} \cup \lbrace \sigma_{ab} \rbrace_{1 \leq a,b \leq n} \right).
\end{equation}
Geometrically, a tangent vector $v = (\dot{P}_n, \ldots, \dot{P}_1) \in T^{1,0} \mathrm{Fl}$ lies in $\mathcal{D}^{1,0}$ if and only if $\dot{P}_j(P_j) \subset P_{j-1}/P_j$ for each $j,$ where we regard $\dot{P}_j$ as a linear map $P_j \to \C^{2n+1}/P_j.$ Equivalently, if $T_r = P_{r} / P_{r+1},$ $\ell = P_1^\perp / P_1,$ then
\begin{equation}\label{eq:superhorizontal}
	\mathcal{D}^{1,0} \cong \left( T_1^* \otimes \ell \right) \oplus \bigoplus_{r=2}^{n} T_r^* \otimes T_{r-1}.
\end{equation}

\section{Totally isotropic surfaces}\label{sect:totallyisotropic}

\indent \indent We now introduce our primary objects of interest. Let $f : \Sigma \to S^{2n}$ be a conformal minimal immersion of a Riemann surface and let $z$ be a locally-defined holomorphic coordinate on $\Sigma.$ Following Calabi \cite{Calabi67}, $f : \Sigma \to S^{2n}$ is called \emph{totally isotropic} if
\begin{equation*}
    \left\langle \frac{\partial^r f}{\partial z^r}, \frac{\partial^s f}{\partial z^s} \right\rangle_\C = 0 \text{ for all } r, s \geq 1.
\end{equation*}
Equivalently, the complex osculating spaces
\begin{equation*}
    (\mathcal O_r)_p = \operatorname{span}_{\C} \left\lbrace         \frac{\partial f}{\partial z}(p), \ldots, \frac{\partial^r f}{\partial z^r}(p) \right\rbrace, \quad r \geq 1,
\end{equation*}
are isotropic for all $p \in \Sigma$. These conditions are easily seen to be independent of the choice of $z.$ A conformal minimal branched immersion is called totally isotropic if it is totally isotropic on its regular locus.

The importance of totally isotropic surfaces is underlined by the following beautiful result of Calabi.

\begin{theorem}[\cite{Calabi67}]\label{thm:calabi}
	Every branched minimal immersion $f: S^2 \to S^{N}$ is totally isotropic in the lowest-dimensional totally geodesic subsphere containing its image.
\end{theorem}

\subsection{Frame reduction}

\indent \indent Let us now assume $f: \Sigma \to S^{2n}$ is a linearly full, totally isotropic branched immersion. Let $B \subset \Sigma$ denote the set of branch points of $f$ and let $\Sigma'= \Sigma \setminus B$ denote the open set on which $f$ is an immersion.

The pullback bundle $f^* \SO(2n+1) \to \Sigma$ may be identified with the set of orthonormal frames for $\R^{2n+1}$ based at points of $\Sigma.$ Define a principal $\SO(2) \times \SO(2n-2)$-bundle $\mathcal{B}^{(1)} \to \Sigma'$ over the regular locus of $f$ by restricting to frames for which $\mathbf{f}_{\overline{1}}$ spans $df(T^{1,0}\Sigma').$ On $\mathcal{B}^{(1)}$ we have the following equations
\begin{equation*}
	\eta_{a} = 0 \:\:\: (2 \leq a \leq n), \quad \eta = \eta_1 \neq 0.
\end{equation*}
The form $\eta$ is type $(1,0)$ for $\Sigma$ and the metric $g_{f}$ and area form $dA$ on $\Sigma'$ induced by $f$ pull back to $\mathcal{B}^{(1)}$ as
\begin{equation*}
	g_{f} = \eta \cdot \overline{\eta}, \quad dA = \tfrac{i}{2} \eta \wedge \overline{\eta}.
\end{equation*}
Differentiating the condition $\eta_a = 0$ and using the structure equations (\ref{eq:structeqsC}) gives
\begin{equation*}
	-\kappa_{a \overline{1}} \wedge \eta - \sigma_{a1} \wedge \overline{\eta} = 0, \quad 2 \leq a \leq n,
\end{equation*}
so Cartan's Lemma implies the existence of $\C$-valued functions $B_a, H_a, C_a$ for $2 \leq a \leq n$ such that
\begin{equation*}
	\begin{aligned}
		\kappa_{a \overline{1}} & = B_a \eta + H_a \overline{\eta}, \\
		\sigma_{a1}             & = H_a \eta + C_a \overline{\eta}. 
	\end{aligned}
\end{equation*}
The functions $H_a$ are the components of the mean curvature vector of $\Sigma$ so, since $f$ is minimal, $H_a = 0$. 
    
The condition $\langle \mathcal{O}_2, \mathcal{O}_2 \rangle_{\C} = 0$ implies the vector $\xi = B_a \mathbf{f}_{\overline{a}} + \overline{C_a} \mathbf{f}_a$ is null (i.e., $\langle \xi, \xi \rangle_{\C} = 0$). On the open set $\Sigma''$ in $\Sigma'$ where $\xi \neq 0$, the $\SO(2n-2)$-action allows us to reduce frames so that $\xi = w_2 \mathbf{f}_{\overline{2}},$ where $w_2$ is a $\C$-valued function. Linear fullness implies $w_2$ is not identically zero (indeed, if $w_2 = 0$, then the structure equations would imply $d(\mathbf{x} \wedge \mathbf{e}_1 \wedge \mathbf{e}_2) = 0$). Let $\mathcal{B}^{(2)} \to \Sigma''$ be the principal $\SO(2)^2 \times \SO(2n-4)$-subbundle of $\mathcal{B}^{(1)}$ defined by this frame restriction. Then, on $\mathcal{B}^{(2)},$ we have the following equations
\begin{equation*}
	\kappa_{2 \overline{1}} = w_2 \eta, \quad \kappa_{a \overline{1}} = 0 \:\: (a \geq 3), \quad \sigma_{a1} = 0 \:\: (a \geq 2).
\end{equation*}
The same argument can be repeated inductively. We summarise the result of this process in the following proposition.

\begin{prop}\label{prop:frameadapt}
	Let $f : \Sigma \to S^{2n}$ be a branched linearly full totally isotropic immersion. There is a dense open set $\Sigma^{\circ} \subset \Sigma$ and a principal $T^n$-subbundle $\mathcal{B}^{(n)} \to \Sigma^\circ$ of the pullback bundle $f^* \SO(2n+1) \to \Sigma^\circ$ on which
	\begin{equation*}
		\begin{aligned}
			\eta_a                    & = 0 \quad (a \geq 2),               \\
			\kappa_{r \overline{r-1}} & = w_r \eta \quad (2 \leq r \leq n), \\
			\kappa_{a \overline{b}}   & = 0 \quad (a > b +1),               \\
			\sigma_{ab}               & = 0 \quad (1 \leq a,b \leq n),      
		\end{aligned}
	\end{equation*}
	for some $\C$-valued functions $w_r,$ $2 \leq r \leq n,$ not identically zero on $\mathcal{B}^{(n)}.$
\end{prop}

We adopt the conventions
\begin{equation*}
	w_1 = 1/\sqrt{2}, \quad w_{n+1} = 0.
\end{equation*}

On $\mathcal{B}^{(n)}$, the semibasic forms are $\eta$ and $\overline{\eta}$, and the connection forms are $\kappa_{r\overline{r}}$ for $1 \leq r \leq n$. Moreover, the equations (\ref{eq:frameeqsC}) reduce to
\begin{equation}\label{eq:dxfcomp}
	\begin{aligned}
		d \mathbf{x}                & = \tfrac{1}{2} \mathbf{f}_{\overline{1}} \eta + \tfrac{1}{2} \mathbf{f}_1 \overline{\eta},                                                                     \\
		d \mathbf{f}_{\overline{1}} & = - \overline{\eta} \mathbf{x} + \kappa_{1 \overline{1}} \mathbf{f}_{\overline{1}} + w_2 \eta \mathbf{f}_{\overline{2}},                                       \\
		d \mathbf{f}_{\overline{r}} & = - \overline{w_r} \overline{\eta} \mathbf{f}_{\overline{r-1}} + \kappa_{r \overline{r}} \mathbf{f}_{\overline{r}} + w_{r+1} \eta \mathbf{f}_{\overline{r+1}}, 
	\end{aligned}
\end{equation}
while the structure equations (\ref{eq:structeqsC}) imply
\begin{equation}\label{eq:structeqsSig}
	\begin{aligned}
		d \eta                    & = - \kappa_{1 \overline{1}} \wedge \eta,                                                                                                  \\
		d \kappa_{r \overline{r}} & = \left( \left\lvert w_r \right\rvert^2 - \left\lvert w_{r+1} \right\rvert^2 \right) \eta \wedge \overline{\eta} \quad (1 \leq r \leq n), \\
		d w_r                     & = z_r \eta + w_r(\kappa_{1 \overline{1}} + \kappa_{r-1 \overline{r-1}} - \kappa_{r \overline{r}}),                                        
	\end{aligned}
\end{equation}
for some $\C$-valued functions $z_r$ on $\mathcal{B}^{(n)}.$

\subsection{The twistor lift}\label{ssect:twistorlift}

\indent \indent On $\Sigma^{\circ},$ define a flag of vector bundles
\begin{equation*}
	\mathcal{P}_r = \mathrm{span}_\C(\mathbf{f}_r, \ldots, \mathbf{f}_n), \quad 1 \leq r \leq n, \quad \mathcal{P}_{n+1} = 0.
\end{equation*}
The equations for $d \mathbf{f}_j$ obtained by conjugating (\ref{eq:dxfcomp}) imply that each $\mathcal{P}_r$ is a holomorphic subbundle of the trivial $\C^{2n+1}$-bundle $\underline{\C}^{2n+1} \to \Sigma^\circ.$ In fact, using the technique of ``filling out zeros'' \cite[Prop. 2.2]{BursWood86}, it is straightforward to show that the $\mathcal{P}_r$ extend over the locus $\Sigma \setminus \Sigma^{\circ}$ to produce a global holomorphic isotropic flag
\begin{equation}\label{eq:holomflag}
	0 = \mathcal{P}_{n+1} \subset \mathcal{P}_n \subset \cdots \subset \mathcal{P}_1 \subset \underline{\C}^{2n+1}
\end{equation}
over $\Sigma.$

Define holomorphic line bundles $\mathcal{T}_r \to \Sigma$ and $\mathcal{G}_r \to \Sigma$, $1 \leq r \leq n$, by
\begin{equation*}
	\mathcal{T}_r = \mathcal{P}_r / \mathcal{P}_{r+1}, \quad \mathcal{G}_r = \mathcal{T}_r^*.
\end{equation*}
The $\C$-bilinear inner product $\langle \cdot, \cdot \rangle_\C$ identifies $\mathbf{f}_{\overline{r}}$ with a unitary frame of $\mathcal{G}_r,$ so we have $\mathcal{G}_r = \mathrm{span}_\C(\mathbf{f}_{\overline{r}})$. The Chern connection of $\mathcal{G}_r$ has connection form $\kappa_{r \overline{r}}$. It will be useful to set
\begin{equation*}
	\mathcal{T}_0 = \mathcal{P}_1^\perp / \mathcal{P}_1, \quad \mathcal{G}_0 = \mathcal{T}_0^*.
\end{equation*}
The position vector $\mathbf{x} = f$ is a section of $\mathcal{P}_1^\perp.$ The structure equation (\ref{eq:dxfcomp}) implies that $\delbar \mathbf{x} = \tfrac{1}{2} \mathbf{f}_1 \overline{\eta} \in \Omega^{0,1}(\mathcal{P}_1)$, so the class $\left[ \mathbf{x} \right] \in \Gamma(\mathcal{T}_0)$ is a nowhere-vanishing holomorphic section of $\mathcal{T}_0,$ giving canonical trivializations of $\mathcal{T}_0$ and $\mathcal{G}_0.$

\begin{definition}
	For $1 \leq r \leq n,$ let
	\begin{equation*}
		\Psi_r : \Sigma \to Z_{n-r+1}, \quad \Psi_r(p) = \mathcal{P}_r(p)
	\end{equation*}
	be the \emph{partial twistor lift} of $\Sigma.$ The \emph{full twistor lift} of $\Sigma$ is the map $\Psi = \Psi_1: \Sigma \to Z.$
\end{definition}
Each $\Psi_r$ is a horizontal holomorphic curve in $Z_{n-r+1}$. Moreover, the flag (\ref{eq:holomflag}) defines a holomorphic map $\widehat{\Psi} : \Sigma \to \mathrm{Fl},$ so that we have the following diagram:
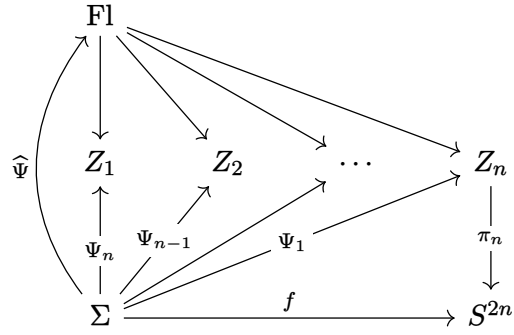
\[\begin{tikzcd}
	{\mathrm{Fl}} &&& \\
	\\
	{Z_1} & {Z_2} & \cdots & {Z_n} \\
	\\
	\Sigma &&& {S^{2n}}
	\arrow[from=1-1, to=3-1]
	\arrow[from=1-1, to=3-2]
	\arrow[from=1-1, to=3-3]
	\arrow[from=1-1, to=3-4]
	\arrow["{\pi_n}"{description}, from=3-4, to=5-4]
	\arrow["{\widehat{\Psi}}", curve={height=-30pt}, from=5-1, to=1-1]
	\arrow["{\Psi_n}"{description, pos=0.4}, from=5-1, to=3-1]
	\arrow["{\Psi_{n-1}}"{description}, from=5-1, to=3-2]
	\arrow[from=5-1, to=3-3]
	\arrow["{\Psi_1}"{description}, from=5-1, to=3-4]
	\arrow["f", from=5-1, to=5-4]
\end{tikzcd}\]
The frame adaptations of Proposition \ref{prop:frameadapt} show $\widehat{\Psi}$ defines an integral submanifold of the holomorphic distribution $\mathcal{D}^{1,0}$ (\ref{eq:superhorizontal}). Since
\begin{equation*}
	\mathcal{D}^{1,0} \cong \bigoplus_{r=1}^{n} \mathcal{T}_r^* \otimes \mathcal{T}_{r-1},
\end{equation*}
the differential $d \widehat{\Psi} \colon T^{1,0}\Sigma \to \mathcal{D}^{1,0}$  has holomorphic components
\begin{equation*}
	\Phi_r \in H^0(\mathcal{K} \otimes \mathcal{T}_r^* \otimes \mathcal{T}_{r-1}) = H^0(\mathcal{K} \otimes \mathcal{G}_{r-1}^* \otimes \mathcal{G}_r),
\end{equation*}
where $\mathcal{K} = \Lambda^{1,0}(T^*\Sigma)$ is the canonical bundle of $\Sigma.$ \\
\indent For $r=1,$ using the canonical trivialization of $\mathcal{T}_0,$ we have
\begin{equation*}
	\Phi_1 = \del f = \tfrac{1}{2} \eta \mathbf{f}_{\overline{1}} \in H^0(\mathcal{K} \otimes \mathcal{G}_1).
\end{equation*}
For $2 \leq r \leq n$ we have the expressions
\begin{equation*}
	\Phi_r = w_r \eta \otimes \mathbf{f}_{\overline{r-1}}^* \otimes \mathbf{f}_{\overline{r}}.
\end{equation*}
In particular, after identifying $\mathcal{T}_r^* = \mathcal{G}_r$ and $\mathcal{T}_{r-1} = \mathcal{G}^{*}_{r-1},$ the structure equations (\ref{eq:dxfcomp}) imply
\begin{equation*}
	d \Psi_r = -\Phi_r \in H^{0}(\mathcal{K} \otimes \mathrm{Hom}(\mathcal{T}_r, \mathcal{T}_{r-1})), \quad 1 \leq r \leq n.
\end{equation*}
Consequently, the zero divisor of $\Phi_r$ is equal to the ramification divisor of $\Psi_r.$ In the particular case of $r = 1$, since $f = \pi_n \circ \Psi_1$ and $d \pi_n$ is an isomorphism on the horizontal distribution $\mathcal{H}_n,$ the zero divisor of $\Phi_1$ is the branch divisor of the branched immersion $f$. Let $B_1$ denote this branch divisor. Then, since $\Phi_1$ defines a holomorphic bundle map $\mathcal{K}^* \to \mathcal{G}_1$ with zero divisor $B_1,$ we have
\begin{equation*}
	\mathcal{G}_1 \cong \mathcal{K}^{*} \otimes \mathcal{O}(B_1).
\end{equation*}

\subsection{Degrees and ramification}\label{ssect:degram}

\indent \indent For $1 \leq r \leq n,$ let $B_r$ denote the zero divisor of $\Phi_r$ and define
\begin{equation*}
	a_r = \deg \mathcal{G}_r, \quad b_r = \deg B_r.
\end{equation*}
Since $\Phi_r$ is a non-zero holomorphic section of $\mathcal{L}_r = \mathcal{K} \otimes \mathcal{G}^{*}_{r-1} \otimes \mathcal{G}_r,$ we have
\begin{equation*}
	b_r = \deg(\mathcal{L}_r) = 2g - 2 - a_{r-1} + a_r, \qquad 1 \leq r \leq n,
\end{equation*}
where $g$ is the genus of $\Sigma$ and we have $a_0 = 0$ since $\mathcal{G}_0$ is trivial. Since $\mathcal{G}_r \cong (\mathcal{K}^{*})^{\otimes r} \otimes \mathcal{L}_1 \otimes \cdots \otimes \mathcal{L}_r$, it follows that
\begin{equation}\label{eq:generaldegbound}
	a_r = 2r(1-g) + \sum_{j=1}^r b_j \geq 2r(1-g).
\end{equation}
In the case where $\Sigma \cong S^2,$ we get
\begin{equation}\label{eq:spheredegbound}
	a_r = 2r + \sum_{j=1}^r b_j \geq 2r.
\end{equation}

The next result, a standard consequence of the classical theory of associated curves \cite{Brycurvenotes19,GriffHarrBook78}, will be useful later. To set notation, let $\mathcal{L}_{\mathrm{spin}}$ denote the pullback of the spinor line bundle $\mathrm{Ca}^* \mathcal{O}_{\mathbb{P}(\Delta_n)}(1) \to Z$ to $\Sigma$. By (\ref{eq:spinorsquarerel}), we have 
\begin{equation}\label{eq:spinorlinesquare}
    \mathcal{L}_{\mathrm{spin}}^{\otimes 2} \cong \det \mathcal{P}_1^* \cong \mathcal{G}_1 \otimes \cdots \otimes \mathcal{G}_n.
\end{equation}

\begin{prop}\label{prop:h0lowerbound}
    We have:
    \begin{enumerate}[label=(\alph*)]
        \item  $h^0(\mathcal{G}_r \otimes \cdots \otimes \mathcal{G}_n) \geq (n-r+1)(n+r)+1.$
        \item $h^0(\mathcal{L}_{\mathrm{spin}}) \geq \tfrac{1}{2}n(n+1)+1.$
    \end{enumerate}
\end{prop}

\begin{proof}
    Let $m = n-r+1 = \mathrm{rk}\,\mathcal{P}_r.$ We have $\mathcal{G}_r \otimes \cdots \otimes \mathcal{G}_n \cong \Lambda^m (\mathcal{P}_r^*).$ Since $\mathcal{P}_r$ is a subbundle of the trivial bundle $\underline{\C}^{2n+1},$ restriction to $\mathcal{P}_r$ defines a map $\mathrm{res}_r: \Lambda^m (\C^{2n+1})^* \to H^0(\mathcal{G}_r \otimes \cdots \otimes \mathcal{G}_n).$ We will prove part (a) by showing this map has rank at least $(n-r+1)(n+r)+1.$

    Consider the $n$th partial twistor lift $\Psi_n : \Sigma \to Z_1 \subset \mathbb{CP}^{2n}.$ The structure equations (\ref{eq:dxfcomp}) show $\Psi_n$ is linearly full (see also \cite{Barbosa75}). The structure equations (\ref{eq:dxfcomp}) also show that under the Pl\"ucker embedding $Z_m \subset \mathbb{P}(\Lambda^m \C^{2n+1}),$ the $r$th partial twistor lift $\Psi_r$ is the $m$th associated curve of $\Psi_n.$ 
    
    Choose a point $p \in \Sigma^\circ$ away from the vanishing loci of the $\Phi_r$ and a basis $v_0, \ldots v_{2n}$ of $\C^{2n+1},$ with
        \begin{equation*}
            v_i = \mathbf{f}_{n-i}(p), \:\: 0 \leq i \leq n-1, \quad
            v_n = \mathbf{x}(p), \quad v_{n+i} = \mathbf{f}_{\overline{i}}(p), \:\: 1 \leq i \leq n.
        \end{equation*}
    for some adapted frame at $p$ (i.e. some element of $\mathcal{B}^{(n)}_p$). Then, since the $\Phi_r(p)$ are non-vanishing, there are locally defined holomorphic functions $h_0, \ldots, h_{2n}$ near $p$ with
    \begin{equation*}
        \Psi_n = \left[ h_0 v_0 + \cdots + h_{2n} v_{2n} \right]
    \end{equation*}
    near $p$ and the order of vanishing of $h_j$ at $p$ exactly equal to $j$ (see also \cite[\S2.1]{Brycurvenotes19}).
    
    The components of the $m$th associated curve $\Psi_r$ in the Pl\"ucker coordinates corresponding to the basis $v_0, \ldots v_{2n}$ are the Wronskians $W(h_{i_1}, \ldots, h_{i_m}),$ $i_1 < \cdots < i_m.$ These Wronskians are local representatives of sections in the image of $\mathrm{res}_r.$ By \cite[\S2.2]{Brycurvenotes19}, the vanishing order of $W(h_{i_1}, \ldots, h_{i_m})$ at $p$ is $i_1 + \cdots + i_m - \tfrac{1}{2}m(m-1)$. As $i_1 < \cdots < i_m$ varies, these vanishing orders include all integers between $0$ and $(n-r+1)(n+r).$ Functions with different vanishing orders are linearly independent, so at least $(n-r+1)(n+r)+1$ of the $W(h_{i_1}, \ldots, h_{i_m})$ are linearly independent, proving (a).

    To prove (b), we will use similar ideas to show that the restriction map $\Delta_n^* \to H^0(\mathcal{L}_{\mathrm{spin}})$ has rank at least $\tfrac{1}{2}n(n+1) + 1.$ Set
    \begin{equation*}
        E = \mathcal{P}_1(p) = \mathrm{span}_\C(v_0, \ldots, v_{n-1})
    \end{equation*}
    and use the model $\Delta_n = \Lambda^\bullet(E^*).$ Let $\xi^1, \ldots, \xi^n$ be the basis of $E^*$ dual to $v_{n-1},\ldots,v_0.$
    Then, the composition $\mathrm{Ca} \circ \Psi : \Sigma \to \mathbb{P}(\Delta_n)$ may be expressed locally as
    \begin{equation*}
        \mathrm{Ca} \circ \Psi = \left[ \sum_{I} \sigma_I \xi^I \right], \quad \text{where} \:\: \xi^I = \bigwedge_{i \in I} \xi^i
    \end{equation*}
    and the sum is taken over subsets $I \subset \lbrace 1, \ldots n \rbrace.$ The $\sigma_I$ are then local representatives of sections in the image of the restriction map. We have $\mathrm{Ca} \circ \Psi(p) = [1] = [\xi^{\emptyset}].$ Next, let $P_I$ denote the Pl\"ucker coordinate obtained by replacing $v_{n-i}$ by $v_{n+i}$ in $v_{0} \wedge \cdots \wedge v_{n-1}$ for each $i \in I.$ Then, a formula of Cartan \cite{Cartan81} (see also \cite[\S1.6-1.7]{Hano96}) gives
    \begin{equation*}
        \frac{\Psi^* P_I}{\Psi^* P_{\emptyset}} = c_I\left(\frac{\sigma_I}{\sigma_\emptyset}\right)^2
    \end{equation*}
    for some non-zero constant $c_I.$ The above argument on Wronskians gives that the vanishing order of $\Psi^* P_I$ at $p$ is $2 \sum_{i \in I} i,$ so the vanishing order of $\sigma_{I}$ is $\sum_{i \in I} i.$ As $I$ varies across all subsets of $\lbrace 1, \ldots, n \rbrace,$ these vanishing orders include all integers between $0$ and $\tfrac{1}{2}n(n+1).$ Therefore, at least $\tfrac{1}{2}n(n+1) + 1$ of the $\sigma_I$ are linearly independent, proving (b).
\end{proof}

\begin{cor}\label{cor:asumbound}
    Suppose $\Sigma$ is a compact Riemann surface with genus $g.$ For $1 \leq r \leq n$, we have
    \begin{equation*}
        \sum_{j=r}^n a_j \geq (n-r+1)(n+r) + \min\lbrace g, (n-r+1)(n+r) \rbrace.
    \end{equation*}
    In particular, $a_n \geq 2n + \min\lbrace 2n, g \rbrace.$ For $r=1,$ we have the stronger bound
    \begin{equation*}
        \sum_{j=1}^n a_j \geq n(n+1) + \min\lbrace n(n+1), 2g \rbrace.
    \end{equation*}
\end{cor}

\begin{proof}
    Let $\mathcal{L} = \mathcal{G}_r \otimes \cdots \otimes \mathcal{G}_n,$ so $\deg \mathcal{L} = \sum_{j=r}^n a_j.$ If $h^1(\mathcal{L}) = 0,$ Riemann--Roch gives
    \begin{equation*}
        \sum_{j=r}^n a_j = h^0(\mathcal{L}) + g - 1 \geq (n-r+1)(n+r) + g
    \end{equation*}
    by Proposition \ref{prop:h0lowerbound}. 
    
    On the other hand, suppose $h^1(\mathcal{L}) > 0.$ Noting that Proposition \ref{prop:h0lowerbound} implies $h^0(\mathcal{L}) > 0,$ so $\mathcal{L}$ is effective, Clifford's theorem on special divisors \cite[\S2.3]{GriffHarrBook78} gives
    \begin{equation*}
        \sum_{j=r}^n a_j \geq 2(h^0(\mathcal{L})-1) \geq 2 (n-r+1)(n+r),
    \end{equation*}
    again by Proposition \ref{prop:h0lowerbound}.

    The $r=1$ bound is proven by applying the same reasoning to $\mathcal{L}_{\mathrm{spin}},$ which has degree $\tfrac{1}{2} \sum_{j=1}^n a_j$ by (\ref{eq:spinorlinesquare}).
\end{proof}

For $n=2,$ Chi--Mo \cite{ChiMo96} use Clifford's Theorem to obtain degree bounds for totally isotropic surfaces in $S^4.$ The proof of Corollary \ref{cor:asumbound} is a generalization of their work.

The following general propositions will also be useful later.

\begin{prop}\label{prop:degreecohomobounds}
Let $\mathcal{L} \to \Sigma$ be a holomorphic line bundle.
	\begin{enumerate}
		\item If $\deg(\mathcal{L}) < 0$, then $H^0(\mathcal{L}) = 0$.
		\item If $\deg(\mathcal{L}) > 2g - 2$, then $H^1(\mathcal{L}) = 0$.
	\end{enumerate}
\end{prop}

\begin{prop} \label{prop:Filtration-General} Let $\mathcal{E} \to \Sigma$ be a holomorphic vector bundle of rank $m-1$ that admits a holomorphic filtration $0 = \mathcal{F}_1 \subset \mathcal{F}_2 \subset \cdots \subset \mathcal{F}_m = \mathcal{E}$ in which each $\mathcal{L}_r = \mathcal{F}_r/\mathcal{F}_{r-1}$ is a holomorphic line bundle for $r \geq 2$.
	\begin{enumerate}[(a)]
		\item We have $\chi(\mathcal{E}) = (m-1)(1-g) + \sum_{r = 2}^m \deg(\mathcal{L}_r)$.  
		\item If $H^1(\mathcal{L}_r) = 0$ for each $r \geq 2$, then $H^1(\mathcal{E}) = 0$.
	\end{enumerate}
\end{prop}

\begin{proof} In both parts, we will use the short exact sequence $0 \to \mathcal{F}_{r-1} \to \mathcal{F}_r \to \mathcal{L}_r \to 0$.
	\begin{enumerate}[(a)]
		\item Using the additivity of the Euler characteristic, together with Riemann-Roch, we calculate
		      $$\chi(\mathcal{E}) = \sum_{r = 2}^m \chi(\mathcal{L}_r) = \sum_{r = 2}^m \left( \deg(\mathcal{L}_r) + (1-g) \right)\!.$$
		\item Suppose $H^1(\mathcal{L}_r) = 0$ for each $r \geq 2$. Since $\mathcal{F}_2 = \mathcal{L}_2$, we have $H^1(\mathcal{F}_2) = 0$. Next, for each $r = 3, \ldots, m$, the short exact sequence above yields a long exact sequence
		      $$\cdots \to H^1(\mathcal{F}_{r-1}) \to H^1(\mathcal{F}_r) \to H^1(\mathcal{L}_r) \to \cdots$$
		      So, if $H^1(\mathcal{F}_{r-1}) = 0$, then $H^1(\mathcal{F}_r) = 0$. Therefore, by induction, we deduce that $H^1(\mathcal{F}_r) = 0$ for each $r = 2, \ldots, m$, and hence $H^1(\mathcal{E}) = 0$. \qedhere
	\end{enumerate}
\end{proof}

\subsection{The holomorphic normal bundle}

\indent \indent  The holomorphic flag (\ref{eq:holomflag}) allows us to extend the normal bundle of $\Sigma$ across the branch points of $f$. Define a smooth real vector bundle $\nu \to \Sigma$ by the condition $\nu \otimes \C = \mathcal{P}_2 \oplus \overline{\mathcal{P}_2}.$ Over the regular locus $\Sigma^\circ,$ we have
\begin{equation*}
	\mathcal{P}_2 \oplus \overline{\mathcal{P}_2} = \mathrm{span}_\C(\mathbf{f}_2, \mathbf{f}_{\overline{2}}, \ldots, \mathbf{f}_n, \mathbf{f}_{\overline{n}}),
\end{equation*}
so, over $\Sigma^\circ,$ $\nu = N \Sigma^\circ.$ The bundle $\nu$ has an orthogonal complex structure $J_\nu$ defined by
\begin{equation*}
	J_{\nu}|_{\mathcal{P}_2} =  -i , \quad J_{\nu}|_{\overline{\mathcal{P}_2}} =  i,
\end{equation*}
and the structure equations (\ref{eq:dxfcomp}) show this is parallel. Let $\mathcal{N}$ denote the $+i$ eigenspace of $J_\nu$ (i.e. $\overline{\mathcal{P}_2}$), and use the complex bilinear form $\langle \cdot, \cdot \rangle_\C$ to identify $\mathcal{N} \cong \mathcal{P}_2^*.$ Over the regular locus $\Sigma^\circ,$
\begin{equation*}
	\mathcal{N} = \mathrm{span}_\C(\mathbf{f}_{\overline{2}}, \ldots, \mathbf{f}_{\overline{n}}).
\end{equation*}
The $(0,1)$-piece of the normal connection gives $\mathcal{N}$ a holomorphic structure. 

Let $\mathcal{F}_r = \ker(\mathcal{P}_2^* \to \mathcal{P}^*_{r+1})$ be the kernel of the restriction map. There is a holomorphic filtration
\begin{equation}\label{eq:normalfilt}
	0 = \mathcal{F}_1 \subset \mathcal{F}_2 \subset \cdots \subset \mathcal{F}_n = \mathcal{N}.
\end{equation}
Over $\Sigma^\circ,$ $\mathcal{F}_r$ is identified with $\operatorname{span}_{\C}(\mathbf{f}_{\overline 2}, \ldots, \mathbf{f}_{\overline{r}}).$  Also define the quotient bundles $\mathcal{Q}_r = \mathcal{N}/\mathcal{F}_{r-1}$ for $2 \leq r \leq n + 1$, which have $\mathrm{rank}_\C(\mathcal{Q}_r) = n+1-r$. We have $\mathcal{Q}_r \cong \mathcal{P}_r^*$ as holomorphic bundles. In particular, $\mathcal{Q}_2 = \mathcal{N}$ and $\mathcal{Q}_n = \mathcal{G}_n.$ The definition of the $\mathcal{Q}_r$ as quotient bundles gives rise to a holomorphic short exact sequence
\begin{equation}\label{eq:Qses}
	0 \to \mathcal{G}_r \to \mathcal{Q}_r \to \mathcal{Q}_{r+1} \to 0.
\end{equation}
As smooth bundles there are identifications
\begin{align}
	\mathcal{N}   & \simeq \mathcal{G}_2 \oplus \cdots \oplus \mathcal{G}_n, \label{eq:Eident} \\
	\mathcal{Q}_r & \simeq \mathcal{G}_r \oplus \cdots \oplus \mathcal{G}_n, \label{eq:Qident} 
\end{align}
but these are not isomorphisms of holomorphic bundles.

Under the identification (\ref{eq:Eident}), we may write a section of $\mathcal{N}$ as
\begin{equation*}
	\sum_{r=2}^{n} A_r \mathbf{f}_{\overline r} = \begin{bmatrix}
	A_2 \\
	\vdots \\
	A_n
	\end{bmatrix}.
\end{equation*}
Let $\partial^{\perp},$ $\overline {\partial^{\perp}}$ be the Dolbeault operators on $\mathcal{N}$ induced by the normal connection, and let $\del_r, \delbar_r$ be the Dolbeault operators for the Chern connection on $\mathcal{G}_r.$  Then, by (\ref{eq:dxfcomp}), we have
\begin{equation}\label{eq:normconnsplit}
	\partial^{\perp} \begin{bmatrix}
	A_2 \\
	A_3 \\
	\vdots \\
	A_n
	\end{bmatrix} = \begin{bmatrix}
	\partial_2 A_2 \\
	\partial_3 A_3 + \Phi_3 A_2 \\
	\vdots \\
	\partial_n A_n + \Phi_n A_{n-1}
	\end{bmatrix}, \quad \overline{\partial^{\perp}} \begin{bmatrix}
	A_2 \\
	A_3 \\
	\vdots \\
	A_n
	\end{bmatrix} = \begin{bmatrix}
	\overline{\partial_2} A_2 - \Phi_3^\dagger A_3 \\
	\overline{\partial_3} A_3 - \Phi_4^\dagger A_4 \\
	\vdots \\
	\overline{\partial_n} A_n
	\end{bmatrix}.
\end{equation}
Equivalently, if $L$ is the lower triangular matrix
\begin{equation}
	L = \begin{bmatrix}
	0 & 0 &  & & \\
	{\Phi_3} & 0 & 0 &  & \\
	& {\Phi_4} & \ddots & \ddots & \\
	& & \ddots & & \\
	& & & {\Phi_n} & 0
	\end{bmatrix},
\end{equation}
and $U = L^{\dagger},$ and if $\partial_{\mathcal{G}} = \bigoplus \partial_r,$ $\overline{\partial_{\mathcal{G}}} = \bigoplus \delbar_r,$ then (\ref{eq:normconnsplit}) may be written
\begin{equation*}
	\partial^{\perp} = \partial_{\mathcal{G}} + L, \qquad \overline{\partial^{\perp}} = \overline{\partial_{\mathcal{G}}} - U.
\end{equation*}
Passing to the quotient $\mathcal{Q}_r$ and using (\ref{eq:Qident}) gives formulas
\begin{equation}\label{eq:Qconn}
	\partial_{\mathcal{Q}_r} \begin{bmatrix}
	A_r \\
	A_{r+1} \\
	\vdots \\
	A_n
	\end{bmatrix} = \begin{bmatrix}
	\partial_r A_r \\
	\partial_{r+1} A_{r+1} + \Phi_{r+1} A_r \\
	\vdots \\
	\partial_n A_n + \Phi_n A_{n-1}
	\end{bmatrix}, \quad \overline{\partial_{\mathcal{Q}_r}} \begin{bmatrix}
	A_r \\
	A_{r+1} \\
	\vdots \\
	A_n
	\end{bmatrix} = \begin{bmatrix}
	\overline{\partial_r} A_r - \Phi_{r+1}^\dagger A_{r+1} \\
	\overline{\partial_{r+1}} A_{r+1} - \Phi_{r+2}^\dagger A_{r+2} \\
	\vdots \\
	\overline{\partial_n} A_n
	\end{bmatrix}.
\end{equation}
In particular, these formulas show directly that $\partial_{\mathcal{Q}_r}$ preserves $\bigoplus_{j=r+1}^{n} \mathcal{G}_j,$ while $\overline{\partial_{\mathcal{Q}_r}}$ preserves $\mathcal{G}_r.$

\subsubsection{Curvature and degree identities}

\indent \indent In general, let $\mathcal{E} \to \Sigma$ be a Hermitian holomorphic vector bundle with Dolbeault operators $\partial,$ $\overline{\partial}.$ Our convention is that the curvature 2-form of $\mathcal{E}$ is $F_{\mathcal{E}} = i \mathfrak{K} \,dA$, where $\mathfrak{K} \in \Gamma(\mathrm{End}(\mathcal{E}))$.
With this convention we have
\begin{equation}\label{eq:curvcommute}
	\partial \overline{\partial} - \overline{\partial} \partial = -\frac{1}{2} \mathfrak{K}
\end{equation}
and
\begin{equation}\label{eq:degform}
	\deg \mathcal{E} = \frac{i}{2 \pi} \int_{\Sigma} i \operatorname{tr}(\mathfrak{K}) \, dA = -\frac{1}{2\pi} \int_{\Sigma} \operatorname{tr} (\mathfrak{K}) \, dA.
\end{equation}
The K\"ahler identities in complex dimension $1$ imply that the Hermitian adjoints $\partial^*$ and $\overline{\partial}^*$ satisfy
\begin{equation*}
	\partial^* = - \overline{\partial}, \qquad \overline{\partial}^* = - \partial.
\end{equation*}
As a result, integrating by parts,
\begin{equation}\label{eq:curvintbyparts}
	\lVert \partial \sigma \rVert^2 - \lVert \overline{\partial} \sigma \rVert^2 = -\frac{1}{2} \int_{\Sigma} \langle \mathfrak{K} \sigma, \sigma \rangle \, d A.
\end{equation}

We now compute the curvature of the bundles already constructed. From (\ref{eq:dxfcomp}) we see that the connection form for $\mathcal{G}_r$ is $\kappa_{r \overline{r}}$ and (\ref{eq:structeqsSig}) implies
\begin{equation}
	\mathfrak{K}_{\mathcal{G}_r} = 2 \left(\lvert \Phi_{r+1} \rvert^2 - \lvert \Phi_{r} \rvert^2 \right), \label{eq:Grcurv}
\end{equation}
recalling the convention $\Phi_{n+1} = 0$ and the fact $\lvert \Phi_1 \rvert^2 = 1/2.$

\begin{prop}
	We have
	\begin{equation*}
		a_r = \frac{1}{\pi} \int_{\Sigma} \lvert \Phi_r \rvert^2 - \lvert \Phi_{r+1} \rvert^2 \, dA \quad 1 \leq r \leq n,
	\end{equation*}
	and
	\begin{align}
		\operatorname{Area}(f) & = 2 \pi \sum_{r=1}^{n} a_r \label{eq:areaformula}  \\
        & = 2 \pi n(n+1)(1-g) + 2 \pi \sum_{j=1}^{n} (n-j+1)b_j, \label{eq:areaformularamified}
	\end{align}
\end{prop}

\begin{proof}
	The first equality is the degree formula (\ref{eq:degform}) applied to $\mathcal{G}_r.$ The second equality follows from taking a telescoping sum. Finally, rewriting the area formula (\ref{eq:areaformula}) in terms of the $b_j$ using (\ref{eq:generaldegbound}) gives the third equality. 
\end{proof}


Using the above proposition together with Corollary \ref{cor:asumbound}, we obtain:

\begin{cor}
\begin{equation}\label{eq:generalarealowerbound}
    \frac{\operatorname{Area}(f)}{\pi} \geq 2n(n+1) + 2 \min\lbrace 2g, n(n+1) \rbrace.
\end{equation}
\end{cor}

Next, we compute the curvatures of the holomorphic normal bundle $\mathcal{N}$ and of the quotient bundles $\mathcal{Q}_r$.

\begin{prop}
	\begin{align}
		\mathfrak{K}_{\mathcal{N}} = - 2 \lvert \Phi_2 \rvert^2 \, \Pi_2,                     \\
		\mathfrak{K}_{\mathcal{Q}_r} = - 2 \lvert \Phi_r \rvert^2 \, \Pi_r, \label{eq:Qkcurv} 
	\end{align}
	where $\Pi_j$ is the orthogonal projection onto $\mathcal{G}_j$ in the splittings (\ref{eq:Eident}), (\ref{eq:Qident}).
\end{prop}

\begin{proof}
	From (\ref{eq:Qconn}), we have formulas $\partial_{\mathcal{Q}_r} = \partial_{\bigoplus_{j \geq r} \mathcal{G}_j} + L_r,$ $\delbar_{\mathcal{Q}_r} = \delbar_{\bigoplus_{j \geq r} \mathcal{G}_j} - U_r,$ where $L_r$ and $U_r$ are the bottom right submatrices of $L$ and $U$ acting on $\mathcal{Q}_r$ in the splitting (\ref{eq:Qident}). By holomorphicity of the $\Phi_r$ we have
	\begin{equation*}
		\begin{aligned}
			\mathfrak{K}_{Q_r} & = \mathrm{diag}(\mathfrak{K}_{\mathcal{G}_r}, \ldots, \mathfrak{K}_{\mathcal{G}_n}) + 2 (L_r U_r - U_r L_r). 
		\end{aligned}
	\end{equation*}
	The commutator of $L_r$ and $U_r$ cancels every diagonal term (\ref{eq:Grcurv}) in $\mathrm{diag}(\mathfrak{K}_{\mathcal{G}_r}, \ldots, \mathfrak{K}_{\mathcal{G}_n})$ except the first one, where it leaves $-2 \lvert \Phi_r \rvert^2.$ 
\end{proof}

Taking the trace of (\ref{eq:Qkcurv}) gives
\begin{equation*}
	\deg \mathcal{Q}_r = \sum_{j=r}^n a_j = \frac{1}{\pi} \int_{\Sigma} \lvert \Phi_r \rvert^2 \, dA.
\end{equation*}
In particular, for $r=2$ this gives, using $a_1 = 2(1-g)+b_1,$
\begin{equation*}
	\deg \mathcal{N} = \sum_{r=2}^{n} a_r = \frac{\operatorname{Area}(f)}{2\pi} - 2(1-g) - b_1.
\end{equation*}
Applying Riemann-Roch,
\begin{equation}\label{eq:EuCharN}
	\chi(\mathcal{N}) = \frac{\operatorname{Area}(f)}{2\pi} + (n-3)(1-g) - b_1.
\end{equation}

\begin{prop}\label{prop:firstcohomvanish}  We have
	\begin{equation}\label{eq:EuCharGQ}
	    \chi(\mathcal{G}_k \otimes \mathcal{Q}_{\ell+1}) = (n-\ell)(1-g) + \sum_{j = \ell+1}^{n} (a_k + a_{j})
	\end{equation}
	for $1 \leq k \leq n$ and $1 \leq \ell \leq n-1$. Furthermore, if $\Sigma \cong S^2$, then  $H^1(\mathcal{G}_k \otimes \mathcal{Q}_{\ell+1}) = 0.$
\end{prop}

\begin{proof} Fix $k$ and $\ell$. From  (\ref{eq:normalfilt}), we have a holomorphic filtration
	$$0 \subset \mathcal{G}_k \otimes (\mathcal{F}_{\ell+1}/\mathcal{F}_{\ell}) \subset \mathcal{G}_k \otimes (\mathcal{F}_{\ell+2}/\mathcal{F}_{\ell}) \subset \cdots \subset \mathcal{G}_k \otimes (\mathcal{F}_n/\mathcal{F}_\ell) = \mathcal{G}_k \otimes \mathcal{Q}_{\ell+1}$$
	in which each successive quotient $\mathcal{L}_r \cong \mathcal{G}_k \otimes \mathcal{G}_{\ell-1+r}$ is a holomorphic line bundle, where $r = 2, \ldots, n-\ell+1$. The Euler characteristic formula then follows from the first part of Proposition \ref{prop:Filtration-General}. 
    
    If $\Sigma \cong S^2$, then $\deg(\mathcal{L}_r) = a_k + a_{\ell-1+r} \geq 2k + 2(\ell-1+r) > -2$, so that $H^1(\mathcal{L}_r) = 0$. The second part of Proposition \ref{prop:Filtration-General} then gives $H^1(\mathcal{G}_k \otimes \mathcal{Q}_{\ell+1}) = 0.$
\end{proof}

\section{The second variation}\label{sect:secondvar}

\indent \indent Recall the second variation of area \cite{Simons68} for an $n$-dimensional immersed minimal submanifold $F: N \to M$ of a Riemannian manifold $(M,g),$

\begin{equation}\label{eq:gensecondvar}
	I_F(U,V) = \int_N  \langle \nabla^{\perp} U, \nabla^{\perp} V \rangle - \langle \mathcal{R} U, V \rangle - \langle \mathcal{B} U, V \rangle \, dA,
\end{equation}
where $U, V$ are sections of the normal bundle of $N$ and, with respect to a local orthonormal frame $e_1, \ldots, e_n$ for $TN,$
\begin{equation*}
	\begin{aligned}
		\mathcal{R} U & = \sum_{i=1}^n \left( R (U, e_i) e_i \right)^{\perp},                             &
		\mathcal{B}U  & = \sum_{i,j=1}^n \langle \mathrm{II}(e_i,e_j), U \rangle \, \mathrm{II}(e_i,e_j), 
	\end{aligned}
\end{equation*}
where $R$ is the curvature tensor of $(M,g)$ and $\mathrm{II}$ is the second fundamental form of the immersion $F: N \to M.$

The logarithmic cut-off argument of \cite[\S3]{Micallef86} shows that (\ref{eq:gensecondvar}) remains valid in the case of a branched immersion once the normal bundle of $N$ is replaced by a ramified extension $\nu$ (see also \cite[Theorem 2.1 and Equation (2.27)]{EjiriMicallef08}). The \emph{Morse index,} $\operatorname{Ind}_\R F,$ of a branched minimal immersion $F: N \to M$ is then defined as the maximum dimension of a subspace of $\Gamma(\nu)$ which is negative-definite for $I_F.$ The \emph{nullity}, $\operatorname{Null}_\R F$, is defined to be the dimension of the kernel of the associated Jacobi operator.

In the case of a linearly full, totally isotropic branched immersion $f: \Sigma \to S^{2n}$ we have
\begin{equation}\label{eq:RBtotisotrop}
	\begin{aligned}
		\mathcal{R} = 2 \, \mathrm{Id}, \qquad \mathcal{B} = 2 \lvert \Phi_2 \rvert^2 \Pi_2, 
	\end{aligned}
\end{equation}
where the first equation follows from the fact $S^{2n}$ has constant sectional curvature $1$ and the second equation follows from unwinding the frame adaptation of Proposition \ref{prop:frameadapt}. Let $I_f^\C$ denote the Hermitian extension of $I_f$ to the holomorphic normal bundle $\mathcal{N},$ then (\ref{eq:gensecondvar}) and (\ref{eq:RBtotisotrop}) together give
\begin{equation}\label{eq:secondvartotisotropic}
	I_f^\C (V,V) = \int_\Sigma \left\lvert \nabla^\perp V \right\rvert^2 - 2 \lvert V \rvert^2 - 2  \lvert \Phi_2 \rvert^2 \lvert \Pi_2 V \rvert^2 \, dA.
\end{equation}

\begin{prop}\label{prop:secondvasrfactor}
	For $V \in \Gamma(\mathcal{N}),$
	\begin{equation}\label{eq:secondvarfactor}
		I^\C_f(V,V) = 4 \left\lVert \delbar^{\perp} V \right\rVert^2 - 2 \left\lVert V \right\rVert^2. 
	\end{equation}
\end{prop}

\begin{proof}
	The decomposition of $\nabla^{\perp}$ into $(1,0)$ and $(0,1)$ pieces gives
	\begin{equation}\label{eq:normcovderiv}
		\left\lvert \nabla^{\perp} V \right\rvert^2 = 2 \left\lvert \del^{\perp} V \right\rvert^2 + 2 \left\lvert \delbar^{\perp} V \right\rvert^2.
	\end{equation}
	On the other hand, the curvature integral identity (\ref{eq:curvintbyparts}) applied to $\mathcal{N}$ gives
	\begin{equation}\label{eq:normcurvid}
		\left\lVert \del^{\perp} V \right\rVert^2 - \left\lVert \delbar^{\perp} V \right\rVert^2 = \int_{\Sigma} \left\lvert \Phi_2 \right\rvert^2 \left\lvert \Pi_2 V \right\rvert^2 \, dA.
	\end{equation}
	Substituting (\ref{eq:normcovderiv}) into (\ref{eq:secondvartotisotropic}) and using (\ref{eq:normcurvid}) gives the result.
\end{proof}

\begin{cor}
	{}
	\begin{equation*}
		\operatorname{Ind}_\R (f) \geq 2 h^0(\mathcal{N})
	\end{equation*}
\end{cor}

\begin{proof}
	If $V \in H^0(\mathcal{N}),$ then (\ref{eq:secondvarfactor}) gives $I^\C_f(V,V) \leq 0$ with equality only when $V = 0.$ Doubling to account for real dimension gives the result.
\end{proof}

This gives the following result due to Ejiri \cite{Ejiri83} in the unbranched case.

\begin{cor}\label{cor:Ejiribound}
	Suppose $\Sigma \cong S^2.$ Then $H^1(\mathcal{N}) = 0$ and consequently
	\begin{equation*}
		\operatorname{Ind}_\R (f) \geq \frac{\operatorname{Area}(f)}{\pi} + 2(n-3) - 2b_1.
	\end{equation*}
\end{cor}

\begin{proof}
	By (\ref{eq:normalfilt}), we have a holomorphic filtration $0 = \mathcal{F}_1 \subset \mathcal{F}_2 \subset \cdots \subset \mathcal{F}_n = \mathcal{N}$ in which each successive quotient $\mathcal{L}_r = \mathcal{F}_r/\mathcal{F}_{r-1} \cong \mathcal{G}_r$ is a holomorphic line bundle. If $\Sigma \cong S^2$, then $\deg(\mathcal{L}_r) = a_r \geq 2r > -2$, so that $H^1(\mathcal{L}_r) = 0$ for each $r \geq 2$. Proposition \ref{prop:Filtration-General} then implies $H^1(\mathcal{N}) = 0$ and the Euler characteristic formula (\ref{eq:EuCharN}) gives
	\begin{equation*}
		h^0(\mathcal{N}) = \chi(\mathcal{N}) = \frac{\operatorname{Area}(f)}{2\pi} + (n-3)(1-g) - b_1. \qedhere
	\end{equation*}
\end{proof}

\section{Index and nullity bounds}\label{sect:bounds}

\indent \indent Motivated by Proposition \ref{prop:secondvasrfactor} and its application to Corollary \ref{cor:Ejiribound}, we now introduce a sequence of quadratic forms beginning with $\tfrac{1}{4} I^{\C}_f.$  These forms will be defined on the holomorphic vector bundles $\mathcal{E}_k,$ $0 \leq k \leq n-1,$ defined by $\mathcal{E}_0 = \mathcal{N},$ which has rank $n-1$ and $\mathcal{E}_k = \mathcal{G}_k \otimes \mathcal{Q}_{k+1}$ for $1 \leq k \leq n-1$, which have $\mathrm{rank}_\C(\mathcal{E}_k) = n - k$. Set the convention $\mathcal{E}_n = 0.$

For $2 \leq k \leq n-1,$ let $\Pi_{>k}: \mathcal{Q}_{k} \to \mathcal{Q}_{k+1}$ be the orthogonal projection associated to the smooth splitting (\ref{eq:Qident}). Define maps $P_k : \Gamma(\mathcal{E}_{k-1}) \to \Omega^{1,0}(\mathcal{E}_{k})$ for $1 \leq k \leq n$ by
\begin{equation*}
	\begin{aligned}
		P_1(V) & = \Phi_1 V,                                              \\
		P_k(V) & = (\Phi_{k} \otimes \Pi_{>k})V, \quad 2 \leq k \leq n-1, \\
		P_n(V) & = 0.                                                     
	\end{aligned}
\end{equation*}

\begin{definition}\label{defn:qkform}
	For $0 \leq k \leq n-1$, we let $q_k$ be the quadratic form on $\Gamma(\mathcal{E}_k)$ defined by
	\begin{equation*}
		\begin{aligned}
			q_k(W) & = \lVert \overline{\partial}_{\mathcal{E}_k} W \rVert^2 - \int_{\Sigma} \lvert \Phi_{k+1} \rvert^2 \, \lvert \Pi_{>k+1} W \rvert^2 \, dA \\
			       & = \lVert \overline{\partial}_{\mathcal{E}_k} W \rVert^2 - \lVert P_{k+1} W \rVert^2.                                                     
		\end{aligned}
	\end{equation*}
\end{definition}
Note that $q_0 = \tfrac{1}{4} I_f^\C$ by Proposition \ref{prop:secondvasrfactor}, since $\mathcal{E}_0 = \mathcal{N}$ and $\lvert \Phi_1 \rvert^2 = 1/2$. Also, $q_{n-1}(W) = \lVert \overline{\partial}_{\mathcal{E}_{n-1}} W \rVert^2 \geq 0$.

\begin{prop}\label{prop:qkequivform}
	For $1 \leq k \leq n-1,$ we have
	\begin{equation*}
		q_k(W) = \lVert \partial_{\mathcal{E}_k} W \rVert^2 - \int_{\Sigma} \lvert \Phi_k \rvert^2 \lvert W \rvert^2 \, dA.
	\end{equation*}
\end{prop}

\begin{proof}
	The curvature integral formula (\ref{eq:curvintbyparts}) applied to $\mathcal{E}_k$ gives
	\begin{equation*}
		\lVert \partial_{\mathcal{E}_k} W \rVert^2 - \lVert \overline{\partial}_{\mathcal{E}_k} W \rVert^2 = -\tfrac{1}{2} \int_{\Sigma} \langle \mathfrak{K}_{{\mathcal{E}_k}} W, W \rangle \, dA.
	\end{equation*}
	The formulas (\ref{eq:Grcurv}) and (\ref{eq:Qkcurv}) give
	\begin{equation*}
		\begin{aligned}
			\mathfrak{K}_{\mathcal{E}_k} & = 2 \left( \lvert \Phi_{k+1} \rvert^2 - \lvert \Phi_k \rvert^2 \right) \mathrm{Id} - 2  \, \lvert \Phi_{k+1} \rvert^2 \Pi_{k+1} \\
			    & = -2 \, \lvert \Phi_k \rvert^2 \, \mathrm{Id} + 2 \lvert \Phi_{k+1} \rvert^2 \, \Pi_{>k+1}.                                     
		\end{aligned}
	\end{equation*}
	Substituting and rearranging completes the proof.
\end{proof}

Regarding $P_k$ as an element of $\Omega^{1,0}(\mathrm{Hom}(\mathcal{E}_{k-1}, \mathcal{E}_k)),$ its pointwise Hermitian adjoint is $P_k^\dagger \in \Omega^{0,1}(\mathrm{Hom}(\mathcal{E}_k, \mathcal{E}_{k-1})).$ Let $\iota_{k+1} : \mathcal{Q}_{k+1} \to \mathcal{Q}_k$ be the inclusion induced by the smooth identification (\ref{eq:Qident}), so that $\iota_{k+1} = \Pi_{>k}^{\dagger}.$ We may express $P_k^\dagger$ as
\begin{equation*}
	\begin{aligned}
		P_1^\dagger & = \Phi_1^\dagger \otimes \mathrm{Id},                          \\
		P_k^\dagger & = \Phi_k^\dagger \otimes \iota_{k+1}, \quad 2 \leq k \leq n-1. 
	\end{aligned}
\end{equation*}

\begin{prop}\label{prop:commutident}
	For $0 \leq k \leq n-2,$ we have the identity
	\begin{equation*}
		(\overline{\partial_{\mathcal{E}_k}})^* \circ P_{k+1}^\dagger = -P_{k+1}^* \circ \partial_{\mathcal{E}_{k+1}},
	\end{equation*}
	where $P_{k+1}^*: \Omega^{1,0}(\mathcal{E}_{k+1}) \to \Gamma(\mathcal{E}_{k})$ is the formal $L^2$ adjoint of $P_{k+1}.$
\end{prop}

\begin{proof}
	The formula (\ref{eq:Qconn}) implies $\partial_{\mathcal{Q}_{k+1}} \circ \iota_{k+2} = \iota_{k+2} \circ \partial_{\mathcal{Q}_{k+2}}.$ Then, for $1 \leq k \leq n-2,$
	\begin{equation*}
		\begin{aligned}
			(\overline{\partial_{\mathcal{E}_k}})^* P_{k+1}^\dagger Z & = - {\partial_{\mathcal{E}_k}}(\iota_{k+2} (\Phi_{k+1}^\dagger \otimes \mathrm{Id}) Z) \\
			    & = -\iota_{k+2} (\Phi_{k+1}^\dagger \otimes \mathrm{Id}) \partial_{\mathcal{E}_{k+1}} Z \\
			    & = -P_{k+1}^* \partial_{\mathcal{E}_{k+1}} Z,                                           
		\end{aligned}
	\end{equation*}
	where we have used the fact that holomorphicity of $\Phi_{k+1}$ implies $\partial \Phi_{k+1}^\dagger = 0.$ For the last equality, since $P_{k+1}$ is zeroth-order the formal adjoint is just contraction with the pointwise Hermitian adjoint. For $k=0$ the same calculation applies with $\iota_{k+2}$ replaced by $\mathrm{Id}$.
\end{proof}

\begin{cor}\label{cor:qdiff}
	For $0 \leq k \leq n-2,$ we have
	\begin{equation}\label{eq:qkdiffid}
		q_k(W) - q_{k+1}(Z)= \lVert \overline{\partial}_{\mathcal{E}_k} W - P_{k+1}^\dagger Z \rVert^2 - \lVert \partial_{\mathcal{E}_{k+1}} Z +  P_{k+1} W \rVert^2.
	\end{equation}
\end{cor}

\begin{proof}
	Definition \ref{defn:qkform} and Proposition \ref{prop:qkequivform} can be expressed as
	\begin{equation*}
		\begin{aligned}
			q_k(W)     & = \lVert \overline{\partial}_{\mathcal{E}_k} W \rVert^2 - \lVert P_{k+1} W \rVert^2,  \\
			q_{k+1}(Z) & = \lVert \partial_{\mathcal{E}_{k+1}} Z \rVert^2 - \lVert P_{k+1}^\dagger Z \rVert^2, 
		\end{aligned}
	\end{equation*}
	respectively. Proposition \ref{prop:commutident} then allows us to complete some squares with no cross terms in $q_k(W) - q_{k+1}(Z)$:
	\begin{equation*}
		\begin{aligned}
			  & \lVert \overline{\partial}_{\mathcal{E}_k} W - P_{k+1}^\dagger Z \rVert^2 - \lVert \partial_{\mathcal{E}_{k+1}} Z +  P_{k+1} W \rVert^2                                  \\
			  & = \lVert \overline{\partial}_{\mathcal{E}_k} W \rVert^2 + \lVert P_{k+1}^\dagger Z \rVert^2 - \lVert \partial_{\mathcal{E}_{k+1}} Z \rVert^2 - \lVert P_{k+1} W \rVert^2 
		\end{aligned}
	\end{equation*}
	because
	\begin{equation*}
		\begin{aligned}
			\llangle P_{k+1}^\dagger Z, \overline{\partial}_{\mathcal{E}_k} W \rrangle + \llangle \partial_{\mathcal{E}_{k+1}} Z, P_{k+1} W \rrangle = \llangle \left( (\overline{\partial_{\mathcal{E}_k}})^* P_{k+1}^\dagger + P_{k+1}^* \partial_{\mathcal{E}_{k+1}} \right) Z, W \rrangle = 0. 
		\end{aligned}
	\end{equation*}
\end{proof}

\begin{rmk}\label{rmk:twistorinterpret}
	The identity (\ref{eq:qkdiffid}) is central to the results in this paper. It admits an interpretation in terms of the twistor geometry of $\Sigma,$ as we now describe.
			
	For $1\leq k\leq n$, set $m=n-k+1$ and consider the partial twistor lift $\Psi_k:\Sigma \to Z_m$ which, as noted in \S\ref{ssect:twistorlift}, is both holomorphic and
	horizontal because $d\Psi_k(T^{1,0}\Sigma)	\subset	\mathcal H_m^{1,0}$. Let $\mathcal N_m$ denote the holomorphic normal bundle of $\Psi_k$ in $Z_m$. Since $\Psi_k$ is horizontal and $\mathcal H_m^{1,0}$ is holomorphic,	there is a holomorphic exact sequence
	\begin{equation*}
		0 \longrightarrow \mathcal N_{\Psi_k}^{\mathrm{hor}} \longrightarrow \mathcal N_m \longrightarrow \mathcal V_k	\longrightarrow 0,
	\end{equation*}
	where
	\begin{equation*}
		\mathcal N_{\Psi_k}^{\mathrm{hor}} = \frac{\Psi_k^*\mathcal H_m^{1,0}}{d\Psi_k(T^{1,0}\Sigma)}
	\end{equation*}
	is the horizontal normal bundle and $\mathcal{V}_k$ is the quotient
	\begin{equation*}
		\mathcal V_k = \Psi_k^*\left(T^{1,0}Z_m/\mathcal H_m^{1,0}\right) \cong	\Lambda^2\mathcal Q_k
	\end{equation*}
	which measures motion transverse to the horizontal distribution of $Z_m$.
			
	The holomorphic exact sequence
	\begin{equation*}
		0 \longrightarrow \mathcal G_k \longrightarrow \mathcal Q_k \longrightarrow \mathcal Q_{k+1} \longrightarrow 0
	\end{equation*}
	induces
	\begin{equation*}
		0 \longrightarrow \mathcal G_k\otimes\mathcal Q_{k+1}	\longrightarrow	\Lambda^2\mathcal Q_k	\longrightarrow	\Lambda^2\mathcal Q_{k+1}	\longrightarrow 0,
	\end{equation*}
	or in other words,
	\begin{equation*}
		0 \longrightarrow \mathcal E_k \longrightarrow \mathcal V_k	\longrightarrow	\mathcal V_{k+1} \longrightarrow 0.
	\end{equation*}
	In particular, $\mathcal V_k$ has a filtration whose successive quotients are $\mathcal E_k,\mathcal E_{k+1},\ldots,\mathcal E_{n-1}$ and this gives rise to a smooth splitting
	\begin{equation}\label{eq:Vident}
		\mathcal{V}_k \simeq \mathcal E_k \oplus \mathcal E_{k+1} \oplus \cdots \oplus \mathcal E_{n-1},
	\end{equation}
	which is simply the splitting (\ref{eq:Qident}) applied to $\mathcal{V}_k \cong \Lambda^2 \mathcal{Q}_k.$
			
	For $1\leq k\leq n-2$, let $\mathcal B_k = \ker\left(\mathcal V_k \to \mathcal V_{k+2}\right).$ Then
	\begin{equation*}
		0 \longrightarrow \mathcal E_k \longrightarrow \mathcal B_k \longrightarrow	\mathcal E_{k+1} \longrightarrow 0
	\end{equation*}
	is a holomorphic exact sequence. With respect to the smooth splitting $\mathcal B_k \simeq \mathcal E_k\oplus\mathcal E_{k+1}$ induced by (\ref{eq:Vident}), the Dolbeault operators of $\mathcal{B}_k$ are
	\begin{equation*}
		\overline{\partial}_{\mathcal B_k} = \begin{bmatrix}
		\overline{\partial}_{\mathcal E_k} & -P_{k+1}^\dagger \\
		0 & \overline{\partial}_{\mathcal E_{k+1}}
		\end{bmatrix}, \qquad \partial_{\mathcal B_k} =	\begin{bmatrix}	\partial_{\mathcal E_k} & 0\\
		P_{k+1} & \partial_{\mathcal E_{k+1}}
		\end{bmatrix},
	\end{equation*}
	where the off-diagonal terms are induced by $\Phi_{k+1}=-d\Psi_{k+1}$ and its pointwise Hermitian adjoint. A straightforward computation gives 
	\begin{equation*}
		\mathfrak{K}_{\mathcal{B}_k} = \begin{bmatrix}
		-2 \lvert \Phi_k \rvert^2 \, \mathrm{Id} & 0 \\
		0 & 2 \lvert \Phi_{k+2} \rvert^2 \, \Pi_{>k+2}
		\end{bmatrix},
	\end{equation*}
	and the identity (\ref{eq:qkdiffid}) is the curvature integral identity (\ref{eq:curvintbyparts}) applied to $\mathcal{B}_k.$
			
	Moreover, this interpretation also gives an intuitive explanation for the two squares in (\ref{eq:qkdiffid}) and their signs. For the component $(W,Z)$ of a normal variation in $\mathcal E_k\oplus\mathcal E_{k+1}$, the expression $\overline{\partial}_{\mathcal E_k}W-P^\dagger_{k+1}Z$ is the $\mathcal{E}_k$-component of the linearized holomorphicity equation. The horizontality condition on $\Psi_k$ is the equation obtained by projecting $d\Psi_k$ to $\mathcal V_k$; the $\mathcal{E}_{k+1}$-component of the $(1,0)$-part of its linearization is $\partial_{\mathcal E_{k+1}}Z+P_{k+1}W.$ Thus
	\begin{equation*}
		q_k(W)-q_{k+1}(Z) =	\left\lVert	\overline{\partial}_{\mathcal E_k}W-P^\dagger_{k+1}Z \right\rVert^2 - \left\lVert \partial_{\mathcal E_{k+1}}Z+P_{k+1}W \right\rVert^2
	\end{equation*}
	may be interpreted as follows: the positive term measures a failure of holomorphicity, while the negative term measures a failure of horizontality.
\end{rmk}

\subsection{The lower index bound}

\indent \indent The identity (\ref{eq:qkdiffid}) allows us to pass from a negative subspace for $q_{k+1}$ to a negative subspace for $q_{k},$ provided a $\delbar$-equation can be solved to eliminate the positive term in (\ref{eq:qkdiffid}). \newpage

\begin{prop}\label{prop:inductivestep}	
    Let $f: \Sigma \to S^{2n}$ be a compact, linearly full, totally isotropic branched immersion. We have
	\begin{equation*}
		\begin{aligned}
			\operatorname{Ind}_\C(q_0) & \geq \operatorname{Ind}_\C(q_{1}) + \chi(\mathcal{N}), \\
			\operatorname{Ind}_\C(q_k) & \geq \operatorname{Ind}_\C(q_{k+1}) + \chi(\mathcal G_k \otimes \mathcal{Q}_{k+2}) - h^1(\mathcal{G}_k \otimes \mathcal{G}_{k+1}), \quad 1 \leq k \leq n-2. 
		\end{aligned}
	\end{equation*}
\end{prop}

\begin{proof}
	For $1 \leq k \leq n-2,$ consider the short exact sequence obtained by tensoring (\ref{eq:Qses}) by $\mathcal{G}_k,$
	\begin{equation}\label{eq:genses}
		\begin{tikzcd}
			0 & {\mathcal{G}_k \otimes \mathcal{G}_{k+1}} & {\mathcal{E}_k} & {\mathcal{G}_k \otimes \mathcal{Q}_{k+2}} & 0
			\arrow[from=1-1, to=1-2]
			\arrow[from=1-2, to=1-3]
			\arrow["{\mathrm{Id} \otimes \Pi_{>k+1}}", from=1-3, to=1-4]
			\arrow[from=1-4, to=1-5]
		\end{tikzcd} 
	\end{equation}
	and let $V_k < H^0(\mathcal{G}_k \otimes \mathcal{Q}_{k+2})$ be the image of the induced map $\mathrm{Id} \otimes \Pi_{>k+1}: H^0(\mathcal{E}_k) \to H^0(\mathcal{G}_k \otimes \mathcal{Q}_{k+2}).$ Choose a $\C$-linear right inverse $R_k: V_k \to H^0(\mathcal{E}_k).$ If $k=0,$ set $V_0 = H^0(\mathcal{E}_0) = H^0(\mathcal{N})$ and let $R_0 : V_0 \to H^0(\mathcal{E}_0)$ be the identity.
	
	Let $U < \Gamma(\mathcal{E}_{k+1})$ be a negative subspace for $q_{k+1}.$ Consider the map $\Omega: U \to H^1(\mathcal{E}_k),$ $\Omega(Z) = \left[P_{k+1}^\dagger Z\right],$ and let $U' = \ker \Omega.$ The codimension of $U'$ in $U$ is at most $h^1(\mathcal{E}_k).$ For every $Z \in U',$ Dolbeault theory implies the equation $\delbar_{\mathcal{E}_k} W = P_{k+1}^\dagger Z$ has a unique solution $W \in \ker(\delbar_{\mathcal{E}_k})^\perp.$ Denote the resulting map $Z \mapsto W$ by $G_k : U' \to \ker(\delbar_{\mathcal{E}_k})^\perp.$
	
	The restriction of $P_{k+1}^\dagger$ to $U$ is injective because, if $u \in U$ is non-zero,
	\begin{equation*}
		\lVert \partial_{\mathcal{E}_{k+1}} u \rVert^2 - \lVert P_{k+1}^\dagger u \rVert^2 = q_{k+1}(u) < 0,
	\end{equation*}
	so $P_{k+1}^\dagger u \neq 0$.
	
	We claim that $G_k(U') \oplus R_k(V_k)$ is a negative subspace for $q_k$. The sum is direct because $R_k(V_k) < H^0(\mathcal{E}_k)$ and the codomain of $G_k$ is $\ker(\delbar_{\mathcal{E}_k})^\perp.$ For non-zero $W = G_k(Z)+R_k(\sigma) \in G_k(U') \oplus R_k(V_k),$ we have
	\begin{equation*}
		q_k(W) = q_{k+1}(Z) - \lVert \partial_{\mathcal{E}_{k+1}} Z +  P_{k+1} W \rVert^2 < 0,
	\end{equation*}
	because if $Z \neq 0$ the first term is strictly negative, and if $Z = 0,$ then $P_{k+1} W = P_{k+1} R_k \sigma = \Phi_{k+1} \sigma \neq 0$ and the last term is strictly negative. Therefore $G_k(U') \oplus R_k(V_k)$ is a negative subspace for $q_k$, so $\operatorname{Ind}_{\C}(q_k) \geq \dim_{\C} G_k(U') + \dim_{\C} V_k$. Taking $U$ maximal (i.e. of dimension $\operatorname{Ind}_\C(q_{k+1})$) and using $\dim U - \dim U' \leq h^1(\mathcal{E}_k)$ and injectivity of $G_k$ yields
	\begin{equation*}
		\operatorname{Ind}_{\C}(q_k) \geq \operatorname{Ind}_\C(q_{k+1}) - h^1(\mathcal{E}_k) + \dim_{\C} V_k.
	\end{equation*}
	
	It remains to calculate $\dim_\C V_k.$ For $k = 0,$  $\dim V_0 = h^0(\mathcal{E}_0).$ For $1 \leq k \leq n-2,$ the long exact sequence associated to (\ref{eq:genses}) gives
	\begin{equation*}
		\dim_{\C}V_k = h^0(\mathcal{E}_k) - h^0(\mathcal{G}_k\otimes\mathcal{G}_{k+1}).
	\end{equation*}
	Hence, by additivity of the Euler characteristic,
	\begin{equation*}
		\begin{aligned}
			\dim_{\C}V_k-h^1(\mathcal{E}_k) & = \chi(\mathcal{E}_k) - h^0(\mathcal{G}_k\otimes\mathcal{G}_{k+1})                          \\
			 & = \chi(\mathcal{G}_k\otimes\mathcal{Q}_{k+2}) - h^1(\mathcal{G}_k\otimes\mathcal{G}_{k+1}).
		\end{aligned}
	\end{equation*}
\end{proof}

\newpage

\begin{theorem}\label{thm:genlowbound}
	Let $f: \Sigma \to S^{2n}$ be a compact, linearly full, totally isotropic branched immersion, where $\Sigma$ has genus $g.$ Then
	\begin{equation*}
		\operatorname{Ind}_\R f \geq (n-2) \frac{\operatorname{Area}(f)}{\pi} + 2a_n + n(n-1)(1-g) - 2 \sum_{k=1}^{n-2} h^1(\mathcal{G}_k \otimes \mathcal{G}_{k+1}).
	\end{equation*}
	Equivalently,
	\begin{equation*}
		\operatorname{Ind}_\R f \geq (n-2)\frac{\operatorname{Area}(f)}{\pi} + n(n+3)(1-g) + 2 \sum_{j=1}^n b_j - 2 \sum_{k=1}^{n-2} h^1(\mathcal{G}_k \otimes \mathcal{G}_{k+1}).
	\end{equation*}
\end{theorem}

\begin{proof}
	We have $q_{n-1}(W) = \lVert \delbar_{\mathcal{E}_{n-1}} W \rVert^2 \geq 0,$ so $\operatorname{Ind}_\C(q_{n-1}) = 0.$  Iterating Proposition \ref{prop:inductivestep} down from $k = n-2$ to $k=0$ therefore gives
	\begin{equation*}
		\operatorname{Ind}_{\C}(q_0) \geq \chi(\mathcal{N}) + \sum_{k=1}^{n-2} \left( \chi(\mathcal{G}_k \otimes \mathcal{Q}_{k+2}) - h^1(\mathcal{G}_k \otimes \mathcal{G}_{k+1}) \right).
	\end{equation*}
	Using the formulas (\ref{eq:EuCharN}) for $\chi(\mathcal{N})$ and (\ref{eq:EuCharGQ}) for $\chi(\mathcal{G}_k \otimes \mathcal{Q}_{k+2}),$ we obtain
	\begin{equation*}
		\begin{aligned}
			\operatorname{Ind}_{\C}(q_0) & \geq \sum_{r=2}^n(a_r+1-g) + \sum_{k=1}^{n-2} \sum_{r=k+2}^{n}(a_k + a_r + 1 -g) - \sum_{k=1}^{n-2} h^1(\mathcal{G}_k \otimes \mathcal{G}_{k+1}) \\
			                             & = \frac{n-2}{2} \frac{\operatorname{Area}(f)}{\pi} + a_n + \frac{n(n-1)}{2}(1-g) - \sum_{k=1}^{n-2} h^1(\mathcal{G}_k \otimes \mathcal{G}_{k+1}). 
		\end{aligned}
	\end{equation*}
	Doubling to account for $\R$ gives the first form of the bound, since $q_0 = \tfrac{1}{4} I^\C_f.$ The second form follows from using (\ref{eq:generaldegbound}) to express $a_n$ as $2n(1-g) + \sum_{j=1}^n b_j.$
\end{proof}

\subsection{The upper index bound}

\indent \indent The symmetry in the identity (\ref{eq:qkdiffid}) suggests reversing the strategy of the previous section to prove upper bounds on the Morse index. For this, we aim to extend a negative subspace for $q_{k}$ to a negative subspace for $q_{k+1}$, which involves solving a $\del$-equation.

\begin{prop}\label{prop:upperbound}
	Let $f: \Sigma \to S^{2n}$ be a compact, linearly full, totally isotropic branched immersion. For $0 \leq k \leq n-2,$ we have
	\begin{equation*}
		\operatorname{Ind}_\C(q_{k}) \leq h^0(\mathcal{K} \otimes \mathcal{E}_{k+1}) + \operatorname{Ind}_\C(q_{k+1}).
	\end{equation*}
\end{prop}

\begin{proof}
	Let $U < \Gamma(\mathcal{E}_k)$ be a negative subspace for $q_k.$ Define a map $\Theta : U \to H^0(\mathcal{K} \otimes \mathcal{E}_{k+1})$ by $\Theta(W) = \pi_{H^0(\mathcal{K} \otimes \mathcal{E}_{k+1})} (P_{k+1} W)$ and let $U' = \ker \Theta.$ The codimension of $U'$ in $U$ is at most $h^0(\mathcal{K} \otimes \mathcal{E}_{k+1}).$
	
	If $W \in U',$ then $P_{k+1} W$ is orthogonal to $\ker(\delbar_{\mathcal{K} \otimes \mathcal{E}_{k+1}}),$ so by Dolbeault theory there is a unique $Z \in \ker (\del_{\mathcal{E}_{k+1}})^\perp$ with $\del_{\mathcal{E}_{k+1}} Z = -P_{k+1} W$. Denote the map $W \mapsto Z$ by $S: U' \to \ker (\del_{\mathcal{E}_{k+1}})^\perp.$ The map $S$ is injective because if $Z = 0,$ then $P_{k+1} W = 0$ and $0 \geq q_k(W) = \lVert \delbar_{\mathcal{E}_k} W \rVert^2 \geq 0,$ so $q_k(W) = 0$ and hence $W = 0$ by the fact $q_{k}$ is negative-definite on $U'.$
	
	We claim that $S(U')$ is a negative subspace for $q_{k+1}.$ This is because for $Z = S(W)$ with $W \in U'$ non-zero we have, by (\ref{eq:qkdiffid}),
	\begin{equation*}
		q_{k+1}(Z) = q_k(W) - \lVert \delbar_{\mathcal{E}_k} W - P_{k+1}^\dagger Z \rVert^2 < 0.
	\end{equation*}
	Therefore $\operatorname{Ind}_\C (q_{k+1}) \geq \dim_{\C}(U') \geq \dim_{\C}(U) - h^0(\mathcal{K} \otimes \mathcal{E}_{k+1}).$ Taking $U$ maximal, i.e. of dimension $\operatorname{Ind}_\C(q_k),$ gives the result.
\end{proof}

\begin{theorem}\label{thm:upperbound}
	Let $f: \Sigma \to S^{2n}$ be a compact, linearly full, totally isotropic branched immersion, where $\Sigma$ has genus $g.$ Then
	\begin{equation*}
		\operatorname{Ind}_\R f \leq (n-1) \left( \frac{\operatorname{Area}(f)}{\pi} - n(1-g) \right) + 2 \sum_{k=0}^{n-2} h^1(\mathcal{K} \otimes \mathcal{E}_{k+1}).
	\end{equation*}
\end{theorem}

\begin{proof}
	Induction from Proposition \ref{prop:upperbound} together with the fact that $\operatorname{Ind}_\C(q_{n-1}) = 0$, since $q_{n-1}(W) = \lVert \delbar_{\mathcal{E}_{n-1}} W \rVert^2,$ implies
	\begin{equation}\label{eq:IndCupperhzero}
		\operatorname{Ind}_\C(q_0) \leq \sum_{k=0}^{n-2} h^0(\mathcal{K} \otimes \mathcal{E}_{k+1}).
	\end{equation}
	The bundle $\mathcal{K} \otimes \mathcal{E}_{k+1}$ has a filtration with successive line bundle quotients $\mathcal{K} \otimes \mathcal{G}_{k+1} \otimes \mathcal{G}_r$ for $r = k+2, \ldots, n,$ so additivity of the Euler characteristic gives
	\begin{equation*}
		\chi(\mathcal{K} \otimes \mathcal{E}_{k+1}) = \sum_{r=k+2}^n(a_{k+1} + a_r + g - 1).
	\end{equation*}
	The right-hand-side of (\ref{eq:IndCupperhzero}) may then be re-written as
	\begin{equation*}
		\begin{aligned}
			    & \sum_{k=0}^{n-2} \left( \chi(\mathcal{K} \otimes \mathcal{E}_{k+1}) + h^1(\mathcal{K} \otimes \mathcal{E}_{k+1}) \right)              \\
			=\, & \sum_{1 \leq j < r \leq n}(a_j+a_r+g-1) + \sum_{k=0}^{n-2} h^1(\mathcal{K} \otimes \mathcal{E}_{k+1})                                 \\
			=\, & \frac{n-1}{2} \left( \frac{\operatorname{Area}(f)}{\pi} - n(1-g) \right) + \sum_{k=0}^{n-2} h^1(\mathcal{K} \otimes \mathcal{E}_{k+1}) 
		\end{aligned}
	\end{equation*}
	where we have used (\ref{eq:areaformula}) to simplify the first two terms of the second line. Doubling to account for $\R$ gives the result, since $q_0 = \tfrac{1}{4} I^\C_f.$ 
\end{proof}

\subsection{Nullity bounds}

\indent \indent The same strategy used to produce the upper and lower index bounds of Theorems \ref{thm:upperbound} and \ref{thm:genlowbound} may be used to produce bounds on the nullity of a compact, linearly full, totally isotropic branched immersion $f: \Sigma \to S^{2n}.$

\begin{definition}
    Let $\mathcal{J}_k$ be the operator on $L^2(\mathcal{E}_k)$ associated to the quadratic form $q_k,$ so that
    \begin{equation*}
        \llangle \mathcal{J}_k W, W \rrangle = q_k(W).
    \end{equation*}
    A \emph{$q_k$-Jacobi field} is a section $W \in L^2(\mathcal{E}_k)$ in $\ker \mathcal{J}_k.$ The \emph{$q_k$-nullity} is the dimension $\nu_k = \dim \ker \mathcal{J}_k$ of the space of $q_k$-Jacobi fields.
\end{definition}

From Definition \ref{defn:qkform} and Proposition \ref{prop:qkequivform}, a $q_k$-Jacobi field satisfies
\begin{equation}\label{eq:qkJacobiidentities}
    \delbar^*_{\mathcal{E}_k} \delbar_{\mathcal{E}_k} W = P_{k+1}^* P_{k+1} W, \quad \del^*_{\mathcal{E}_k} \del_{\mathcal{E}_k} W = (P_k^\dagger)^* P^\dagger_k W.
\end{equation}

Holomorphicity of $\Phi_{k+1}$ implies $P_{k+1} H^0(\mathcal{E}_k) < H^0(\mathcal{K} \otimes \mathcal{E}_{k+1}),$ while Proposition \ref{prop:commutident} implies $P^\dagger_{k+1} \ker(\del_{\mathcal{E}_{k+1}}) < H^1(\mathcal{E}_{k}).$ Let
\begin{equation}
    c_k = \dim_\C \frac{H^0(\mathcal{K} \otimes \mathcal{E}_{k+1})}{P_{k+1} H^0(\mathcal{E}_k)}, \quad d_k = \dim_\C \frac{H^1(\mathcal{E}_{k})}{P^\dagger_{k+1} \ker(\del_{\mathcal{E}_{k+1}})}.
\end{equation}
Let $\mathcal{A}_k$ denote the bundle $\mathcal{G}_k \otimes \mathcal{G}_{k+1}$ for $1 \leq k \leq n-1$ and set $\mathcal{A}_0 = 0.$

\begin{lemma}\label{lem:ckdkcalc}
    Let $f: \Sigma \to S^{2n}$ be a compact, linearly full, totally isotropic branched immersion. For $0 \leq k \leq n-2,$ we have
    \begin{equation*}
        c_k + d_k = (n-k-1)b_{k+1} + h^1(\mathcal{A}_k).
    \end{equation*}
\end{lemma}

\begin{proof}
    Since multiplication by $\Phi_{k+1}$ is injective, the kernel of $P_{k+1}$ restricted to $H^0(\mathcal{E}_k)$ is equal to the kernel of $\mathrm{Id} \otimes \Pi_{>k+1}$ restricted to $H^0(\mathcal{E}_k),$ which is $H^0(\mathcal{A}_k).$ Therefore,
    \begin{equation*}
        c_k = h^0(\mathcal{K} \otimes \mathcal{E}_{k+1}) - h^0(\mathcal{E}_k) + h^0(\mathcal{A}_k).
    \end{equation*}
    Similarly, the map $P_{k+1}^\dagger$ is injective on $\Gamma(\mathcal{E}_{k+1}),$ so, by Serre duality,
    \begin{equation*}
        d_k =h^1(\mathcal{E}_k) - h^0(\mathcal{E}^*_{k+1}) = h^1(\mathcal{E}_k) - h^1(\mathcal{K} \otimes \mathcal{E}_{k+1}).
    \end{equation*}
    Combining terms, we have
    \begin{equation*}
        c_k + d_k = \chi(\mathcal{K} \otimes \mathcal{E}_{k+1}) - \chi(\mathcal{E}_k) + \chi(\mathcal{A}_k) + h^1(\mathcal{A}_k).
    \end{equation*}
    Applying Riemann-Roch to each Euler characteristic, the rank terms cancel and we get
    \begin{equation*}
        c_k + d_k = \deg (\mathcal{K} \otimes \mathcal{E}_{k+1}) - \deg (\mathcal{E}_k) + \deg (\mathcal{A}_k) + h^1(\mathcal{A}_k).
    \end{equation*}
    The short exact sequence
    \begin{equation*}
        0 \to \mathcal{A}_k \to \mathcal{E}_k \to \mathcal{G}_k \otimes \mathcal{Q}_{k+2} \to 0
    \end{equation*}
    and the isomorphism $\mathcal{K} \otimes \mathcal{E}_{k+1} \cong (\mathcal{G}_k \otimes \mathcal{Q}_{k+2}) \otimes (\mathcal{K} \otimes \mathcal{G}_k^* \otimes \mathcal{G}_{k+1})$ give
    \begin{equation*}
    \begin{aligned}
        \deg(\mathcal{K} \otimes \mathcal{E}_{k+1}) &= \deg (\mathcal{G}_k \otimes \mathcal{Q}_{k+2}) + (n-k-1) \deg (\mathcal{K} \otimes \mathcal{G}_k^* \otimes \mathcal{G}_{k+1}) \\
        &= \deg(\mathcal{E}_k) - \deg(\mathcal{A}_k) + (n-k-1) b_{k+1},
    \end{aligned}
    \end{equation*}
    completing the proof.
\end{proof}

The next proposition is the analogue for the nullity of Propositions \ref{prop:inductivestep} and \ref{prop:upperbound}. The method of proof is essentially the same.

\begin{prop}\label{prop:nullityinduction}
    Let $f: \Sigma \to S^{2n}$ be a compact, linearly full, totally isotropic branched immersion. For $0 \leq k \leq n-2,$
    \begin{enumerate}[label=(\alph*)]
        \item \begin{equation}
            \nu_k \geq \nu_{k+1} + \chi(\mathcal{A}_k) - (n-k-1)b_{k+1}
        \end{equation}
        \item
            \begin{equation}
                \nu_k \leq \nu_{k+1} + h^0(\mathcal{A}_k) + h^1(\mathcal{A}_k) + (n-k-1)b_{k+1}
            \end{equation}
    \end{enumerate}
\end{prop}

\begin{proof}
    Any pair $(W,Z) \in \Gamma(\mathcal{E}_k) \oplus \Gamma(\mathcal{E}_{k+1})$ satisfying the coupled system of equations
        \begin{equation}\label{eq:nullityextension}
            \delbar_{\mathcal{E}_{k}} W - P^\dagger_{k+1} Z = 0, \quad \del_{\mathcal{E}_{k+1}} Z + P_{k+1} W = 0.
        \end{equation}
    consists of a pair of $q_k$- and $q_{k+1}$-Jacobi fields. Indeed, by Proposition \ref{prop:commutident}, if $(W,Z)$ solves (\ref{eq:nullityextension}), then
        \begin{equation*}
            \begin{aligned}
                \mathcal{J}_k W &= \delbar_{\mathcal{E}_k}^* \delbar_{\mathcal{E}_{k}} W - P_{k+1}^* P_{k+1} W \\
                &= \left( \delbar_{\mathcal{E}_k}^* P^\dagger_{k+1} + P_{k+1}^* \del_{\mathcal{E}_{k+1}} \right)Z = 0.
            \end{aligned}
        \end{equation*}
        The calculation showing $\mathcal{J}_{k+1} Z = 0$ is similar.
    
        Let $U_{k} = \ker \mathcal{J}_k$ denote the space of $q_{k}$-Jacobi fields. For part (a), given $Z \in U_{k+1},$ we will try to construct $W \in \Gamma(\mathcal{E}_k)$ satisfying (\ref{eq:nullityextension}), while for part (b), given $W \in U_k,$ we will try to construct $Z \in \Gamma(\mathcal{E}_{k+1})$ satisfying (\ref{eq:nullityextension}).

        \begin{enumerate}[label=(\alph*)]
        \item We solve $\delbar_{\mathcal{E}_{k}} W - P^\dagger_{k+1} Z = 0$ first. In order to solve this equation, we need the class $[P_{k+1}^\dagger Z] \in H^1(\mathcal{E}_k)$ to vanish. Let $U'_{k+1}$ denote the subspace of $U_{k+1}$ for which this class vanishes. If $Z$ is a $q_{k+1}$-Jacobi field and if $\sigma \in \ker(\del_{\mathcal{E}_{k+1}}),$
        \begin{equation*}
            \begin{aligned}
                \llangle P_{k+1}^\dagger Z, P_{k+1}^\dagger \sigma \rrangle &= \llangle (P_{k+1}^\dagger)^* P_{k+1}^\dagger Z, \sigma \rrangle \\
                &= \llangle \del_{\mathcal{E}_{k+1}} Z, \del_{\mathcal{E}_{k+1}} \sigma \rrangle = 0,
            \end{aligned}
        \end{equation*}
        where we have used (\ref{eq:qkJacobiidentities}). Therefore, $P_{k+1}^\dagger U_{k+1}$ is orthogonal to $P_{k+1}^\dagger \ker(\del_{\mathcal{E}_{k+1}}),$ so the codimension of $U_{k+1}'$ in $U_{k+1}$ is at most $d_k.$

        For $Z \in U_{k+1}',$ choose the solution $W$ of $\delbar_{\mathcal{E}_{k}} W - P^\dagger_{k+1} Z = 0$ with $W \in \ker(\delbar_{\mathcal{E}_{k}} )^\perp$ and let $Y = \del_{\mathcal{E}_{k+1}} Z + P_{k+1} W.$ Proposition \ref{prop:commutident} and (\ref{eq:qkJacobiidentities}) give
        \begin{equation*}
            \begin{aligned}
                \del^*_{\mathcal{E}_{k+1}} Y &= \del^*_{\mathcal{E}_{k+1}} \del_{\mathcal{E}_{k+1}} Z + \del^*_{\mathcal{E}_{k+1}} P_{k+1} W \\
                &= (P_{k+1}^\dagger)^* P_{k+1}^\dagger Z + \del^*_{\mathcal{E}_{k+1}} P_{k+1} W \\
                &= \left( (P_{k+1}^\dagger)^* \delbar_{\mathcal{E}_{k}} + \del^*_{\mathcal{E}_{k+1}} P_{k+1} \right) W = 0,
            \end{aligned}
        \end{equation*}
        so $Y \in H^0(\mathcal{K} \otimes \mathcal{E}_{k+1}).$ Replacing $W$ by $W+\tau$ with $\tau \in H^0(\mathcal{E}_k)$ changes $Y$ by $P_{k+1} \tau.$ Let $U''_{k+1}$ denote the subspace of $U_{k+1}'$ for which the class of $Y$ in ${H^0(\mathcal{K} \otimes \mathcal{E}_{k+1})}/{P_{k+1} H^0(\mathcal{E}_k)}$ is $0.$ The codimension of $U''_{k+1}$ in $U'_{k+1}$ is at most $c_k.$ Given $Z \in U_{k+1}'',$ we can find $W' \in \Gamma(\mathcal{E}_k)$ satisfying both equations in (\ref{eq:nullityextension}), unique up to an element in $\ker(\delbar_{\mathcal{E}_k}) \cap \ker(P_{k+1}) = H^0(\mathcal{A}_k).$ Therefore, by Lemma \ref{lem:ckdkcalc},
        \begin{equation}
            \nu_k \geq \nu_{k+1} - c_k - d_k + h^0(\mathcal{A}_k) = \nu_{k+1} + \chi(\mathcal{A}_k) - (n-k-1)b_{k+1}.
        \end{equation}
        
        \item We solve $\del_{\mathcal{E}_{k+1}} Z + P_{k+1} W = 0$ first. In order to solve this equation, the orthogonal projection of $P_{k+1}W$ onto $H^0(\mathcal K \otimes \mathcal E_{k+1})$ must vanish. Let $V'_{k} < U_k$ denote the subspace of elements $W \in U_k$ for which this projection of $P_{k+1}W$ vanishes. If $W$ is a $q_k$-Jacobi field and if $\alpha \in H^0(\mathcal{E}_k),$ then
        \begin{equation*}
        \begin{aligned}
             \llangle P_{k+1} W, P_{k+1} \alpha \rrangle  &= \llangle P_{k+1}^* P_{k+1} W, \alpha \rrangle \\
             &= \llangle \delbar_{\mathcal{E}_k} W, \delbar_{\mathcal{E}_k} \alpha \rrangle = 0,
        \end{aligned}
        \end{equation*}
        where we have used (\ref{eq:qkJacobiidentities}). Therefore, $P_{k+1} U_k$ is orthogonal to $P_{k+1} H^0(\mathcal{E}_k),$ so the codimension of $V_k'$ in $U_k$ is at most $c_k.$

        For $W \in V_{k}',$ choose the solution $Z$ of $\del_{\mathcal{E}_{k+1}} Z + P_{k+1} W = 0$ with $Z \in \ker (\partial_{\mathcal E_{k+1}})^\perp$ and let $X = \delbar_{\mathcal{E}_{k}} W - P^\dagger_{k+1} Z.$ Proposition \ref{prop:commutident} and (\ref{eq:qkJacobiidentities}) give
        \begin{equation*}
            \begin{aligned}
                \delbar_{\mathcal{E}_k}^* X &= \delbar_{\mathcal{E}_k}^* \delbar_{\mathcal{E}_{k}} W - \delbar_{\mathcal{E}_k}^* P^\dagger_{k+1} Z \\
                &= P_{k+1}^* P_{k+1} W - \delbar_{\mathcal{E}_k}^* P^\dagger_{k+1} Z \\
                &= -\left( P_{k+1}^* \del_{\mathcal{E}_{k+1}} +  \delbar_{\mathcal{E}_k}^* P^\dagger_{k+1} \right) Z = 0,
            \end{aligned}
        \end{equation*}
        so $X \in H^1(\mathcal{E}_k).$ Replacing $Z$ by $Z + \gamma$ with $\gamma \in \ker (\partial_{\mathcal E_{k+1}})$ changes $X$ by $-P_{k+1}^\dagger \gamma.$ Let $V_{k}''$ denote the subspace of $V_k'$ for which the class of $X$ in ${H^1(\mathcal{E}_{k})}/({P^\dagger_{k+1} \ker(\del_{\mathcal{E}_{k+1}})})$ is $0.$ The codimension of $V_{k}''$ in $V_k'$ is at most $d_k.$ Given $W \in V_{k}'',$ we can find $Z' \in \Gamma(\mathcal{E}_{k+1})$ satisfying both equations in (\ref{eq:nullityextension}). The resulting map $V_{k}'' \to U_{k+1}$ has kernel contained in $H^0(\mathcal{A}_k),$ since if $Z' = 0$ then (\ref{eq:nullityextension}) gives $\delbar_{\mathcal{E}_k} W = 0$ and $P_{k+1} W = 0,$ so $W \in \ker(\delbar_{\mathcal{E}_k}) \cap \ker(P_{k+1}) = H^0(\mathcal{A}_k).$ Since $V_k''$ has codimension at most $c_k+d_k$ in $U_k,$ we have, using Lemma \ref{lem:ckdkcalc},
        \begin{equation*}
            \nu_{k+1} \geq \nu_k - c_k - d_k - h^0(\mathcal{A}_k) = \nu_k -h^0(\mathcal{A}_k) - h^1(\mathcal{A}_k) - (n-k-1)b_{k+1}.\qedhere
        \end{equation*}
    \end{enumerate}
\end{proof}

\begin{theorem}\label{thm:nullitybound}
    Let $f: \Sigma \to S^{2n}$ be a compact, linearly full, totally isotropic branched minimal immersion, where $\Sigma$ has genus $g.$ Then
    \begin{equation}
        \left\lvert \operatorname{Null}_\R(f) - 2 \sum_{k=1}^{n-1} h^0(\mathcal{G}_k \otimes \mathcal{G}_{k+1}) \right\rvert \leq 2 \sum_{j=1}^{n-1} (n-j) b_j + 2 \sum_{k=1}^{n-2} h^1(\mathcal{G}_k \otimes \mathcal{G}_{k+1}).
    \end{equation}
\end{theorem}

\begin{proof}
    Since $q_0 = \tfrac{1}{4} I^\C_f,$ $\operatorname{Null}_\R(f) = 2 \nu_0.$ On the other hand, $q_{n-1}(W) = \lVert \delbar_{\mathcal{E}_{n-1}} W \rVert^2,$ so $\nu_{n-1} = h^0(\mathcal{E}_{n-1}) = h^0(\mathcal{G}_{n-1} \otimes \mathcal{G}_n).$ The result then follows by induction on the bounds of Proposition \ref{prop:nullityinduction}.
\end{proof}

\begin{prop}\label{prop:sumh0gkgk1}
    We have
    \begin{equation*}
        \sum_{k=1}^{n-1} h^0(\mathcal{G}_k \otimes \mathcal{G}_{k+1}) = \frac{\operatorname{Area}(f)}{2\pi} + (n^2-3)(1-g)-b_1 + \sum_{j=1}^{n-1} (n-j)b_j + \sum_{k=1}^{n-1} h^1(\mathcal{G}_k \otimes \mathcal{G}_{k+1}).
    \end{equation*}
\end{prop}

\begin{proof}
    By Riemann-Roch, $h^0(\mathcal{G}_k \otimes \mathcal{G}_{k+1}) = a_k + a_{k+1} + 1 - g + h^1(\mathcal{G}_k \otimes \mathcal{G}_{k+1}),$ so
    \begin{equation*}
        \begin{aligned}
            \sum_{k=1}^{n-1} h^0(\mathcal{G}_k \otimes \mathcal{G}_{k+1}) &= \sum_{k=1}^{n-1} (a_k + a_{k+1}) + (n-1)(1-g) + \sum_{k=1}^{n-1} h^1(\mathcal{G}_k \otimes \mathcal{G}_{k+1}) \\
            &= \frac{\operatorname{Area}(f)}{2\pi} + (n^2-3)(1-g)-b_1 + \sum_{j=1}^{n-1} (n-j)b_j + \sum_{k=1}^{n-1} h^1(\mathcal{G}_k \otimes \mathcal{G}_{k+1}),
        \end{aligned}
    \end{equation*}
    where in the final equality we have used the area formula (\ref{eq:areaformula}) and (\ref{eq:generaldegbound}) to rewrite the sums involving $a_k$ and $a_{k+1}$ respectively.
\end{proof}

\section{Consequences in low genus}\label{sect:lowgenus}

\indent \indent We begin by reorganizing the bounds of \S\ref{sect:bounds}. Let
\begin{equation}
\begin{aligned}
    \mathrm{M}(f) &= (n-1)\left(\frac{\operatorname{Area}(f)}{\pi} - n(1-g) \right), \\
    \mathrm{N}(f) &= \frac{\operatorname{Area}(f)}{\pi} + 2(n^2-3)(1-g) -2b_1 + 2 h^1(\mathcal{G}_{n-1} \otimes \mathcal{G}_n),\\
    \mathrm{R}(f) &= \sum_{j=1}^{n-1} (n-j)b_j, \\
    \sigma_-(f) &= \sum_{k=1}^{n-2} h^1(\mathcal{G}_k \otimes \mathcal{G}_{k+1}), \\
    \sigma_+(f) &= \sum_{k=0}^{n-2} h^1(\mathcal{K} \otimes \mathcal{E}_{k+1}).
\end{aligned}
\end{equation}

Combining the bounds of Theorems \ref{thm:genlowbound} and \ref{thm:upperbound} and rewriting the area in terms of the $b_j$ using (\ref{eq:areaformularamified}) we get, for a linearly full, totally isotropic, branched minimal immersion $f: \Sigma \to S^{2n},$
\begin{equation}\label{eq:twosidedbound}
    \mathrm{M}(f) - 2 \mathrm{R}(f) - 2 \sigma_-(f) \leq \operatorname{Ind}_\R f \leq \mathrm{M}(f) + 2 \sigma_+(f).
\end{equation}
From this formula we see that ramification up to index $n-1$ decreases the lower bound, and specialness of the bundles $\mathcal{G}_k \otimes \mathcal{G}_{k+1}$ and $\mathcal{K} \otimes \mathcal{E}_{k+1}$ in the sense of Brill--Noether theory decrease the lower bound and increase the upper bound respectively.

Similarly, using this notation the nullity bound of Theorem \ref{thm:nullitybound} becomes, after using Proposition \ref{prop:sumh0gkgk1},
\begin{equation}\label{eq:twosidednullity}
    \mathrm{N}(f) \leq \operatorname{Null}_\R(f) \leq \mathrm{N}(f) + 4 \mathrm{R}(f) + 4 \sigma_-(f).
\end{equation}
Thus ramification up to index $n-1$ and specialness of $\mathcal{G}_k \otimes \mathcal{G}_{k+1}$ increase the upper bound on the nullity.

\subsection{The $g=0$ case}

\indent \indent Our bounds are strongest in the case $g=0,$ because in this case the $h^1$ terms $\sigma_-(f)$ and $\sigma_+(f)$ vanish for degree reasons. We also recall Calabi's Theorem \ref{thm:calabi} which allows us to reduce the hypotheses on our results from ``totally isotropic'' to ``minimal'' in this case.

\begin{theorem}\label{thm:S2lowerbound}
	Let $f: S^2 \to S^{2n}$ be a linearly full, branched minimal immersion. Then
	\begin{equation}\label{eq:S2lower}
		\operatorname{Ind}_\R f \geq (n-2) \frac{\operatorname{Area}(f)}{\pi} + n(n+3),
	\end{equation}
	with equality if and only if $f$ is twistor-equivalent to the Bor\r{u}vka immersion.
\end{theorem}

\begin{proof} 
	Applying Theorem \ref{thm:genlowbound} to the case $g=0,$ the $h^1$ terms vanish because $\deg (\mathcal{G}_r \otimes \mathcal{G}_{r+1}) = a_r+a_{r+1} \geq 4r+2 > -2.$ Therefore,
	\begin{equation}\label{eq:lowerboundineqs}
		\begin{aligned}
			\operatorname{Ind}_\R f & \geq (n-2) \frac{\operatorname{Area}(f)}{\pi} + n(n-1) + 2a_n \\
			                  & \geq (n-2) \frac{\operatorname{Area}(f)}{\pi} + {n(n+3)}.     
		\end{aligned}
	\end{equation}
	
	In the case of equality we must have equality throughout (\ref{eq:lowerboundineqs}), so $a_n = 2n$ and, by (\ref{eq:spheredegbound}), $b_j = 0$ for $1 \leq j \leq n.$ Thus, since $b_1 = 0,$ the immersion $f$ is unbranched. Furthermore, $a_r = 2r$ for $1 \leq r \leq n$ and by (\ref{eq:areaformula}) the area of $\Sigma$ is $2 \pi n(n+1).$ A result of Barbosa \cite{Barbosa75} then implies that $f(S^2)$ is twistor-equivalent to the Bor\r{u}vka sphere. For the converse, Ball and Madnick \cite{BallMadnick26} proved that the index of the Bor\r{u}vka sphere equals $n(n-1)(2n+1)$ and Karpukhin \cite[Cor 3.10]{Karpukhin21} proved that the index is preserved under twistor deformations.
\end{proof}

\begin{cor}\label{cor:S2boruvlowerindex}
	Let $f: S^2 \to S^{2n}$ be a linearly full, branched minimal immersion. Then
	\begin{equation}\label{eq:niceS2lower}
		\operatorname{Ind}_\R f \geq n(n-1)(2n+1),
	\end{equation}
	with equality if and only if $f$ is twistor-equivalent to the Bor\r{u}vka immersion.
\end{cor}

\begin{proof}
	The area bound (\ref{eq:generalarealowerbound}) for $g=0$ implies $\operatorname{Area}(f) \geq 2 \pi n(n+1),$ so the inequality of Theorem \ref{thm:S2lowerbound} implies (\ref{eq:niceS2lower}). The proof of the equality statement of Theorem \ref{thm:S2lowerbound} shows that in the equality case $\operatorname{Area}(f) = 2 \pi n(n+1),$ so the right hand side of (\ref{eq:S2lower}) is $n(n-1)(2n+1).$
\end{proof}

\begin{theorem} \label{thm:S2upperbound}
	Let $f: S^2 \to S^{2n}$ be a linearly full, branched minimal immersion. Then
	\begin{equation}\label{eq:S2upperM}
		\begin{aligned}
			\operatorname{Ind}_\R f & \leq ({n-1}) \left(\frac{\operatorname{Area}(f)}{\pi} - n \right). 
		\end{aligned}
	\end{equation}
    If, furthermore, $f$ and the intermediate twistor lifts $\Psi_1, \ldots, \Psi_{n-1}$ are unbranched, then
    \begin{equation*}
        \operatorname{Ind}_\R f = (n-1)\left(\frac{\operatorname{Area}(f)}{\pi} - n\right).
    \end{equation*}
\end{theorem} 

\begin{proof}
	The line bundle quotients $\mathcal{K} \otimes \mathcal{G}_{k+1} \otimes \mathcal{G}_r$ of $\mathcal{K} \otimes \mathcal{E}_{k+1}$ have degree $a_{k+1}+a_r-2 \geq 2k + 2r > -2$ by (\ref{eq:spheredegbound}). Therefore, by Proposition \ref{prop:Filtration-General}, each $h^1(\mathcal{K} \otimes \mathcal{E}_{k+1})$ in the bound of Theorem \ref{thm:upperbound} vanishes. This gives $\sigma_+(f) = 0,$ so the upper bound in (\ref{eq:twosidedbound}) gives (\ref{eq:S2upperM}).

    If $f$ and the $\Psi_1, \ldots, \Psi_{n-1}$ are unbranched, then we have $b_1 = \cdots = b_{n-1} = 0,$ so $\mathrm{R}(f) = 0.$ The argument of Theorem \ref{thm:S2lowerbound} implies $\sigma_-(f) = 0,$ so we get equality $\operatorname{Ind}_\R f = \mathrm{M}(f)$ throughout (\ref{eq:twosidedbound}).
\end{proof}

The equality statement of Theorem \ref{thm:S2upperbound} is more general than that of Theorem \ref{thm:S2lowerbound}, since there do exist minimal immersions $f: S^2 \to S^{2n}$ with $b_j = 0$ for $1 \leq j \leq n-1$ but $b_n \neq 0,$ for example those of Barbosa \cite[\S7]{Barbosa75} (see also \cite{Ejiri86}). It would be interesting to determine whether the locus $\operatorname{R}(f)=0$ is dense in each component of the fixed-degree moduli space \cite{Fernandez12}.

\begin{theorem}\label{thm:S2NullityTheorem}
    Let $f: S^{2} \to S^{2n}$ be a linearly full, branched minimal immersion. Then
    \begin{equation*}
        \frac{\operatorname{Area}(f)}{\pi}+2(n^2 - 3) - 2 b_1 \leq \operatorname{Null}_\R(f) \leq \frac{\operatorname{Area}(f)}{\pi}+2(n^2 - 3) - 2 b_1 + 4 \mathrm{R}(f).
    \end{equation*}
    If, furthermore, $f$ and the intermediate twistor lifts $\Psi_1, \ldots, \Psi_{n-1}$ are unbranched, then
    \begin{equation*}
        \operatorname{Null}_\R f = \frac{\operatorname{Area}(f)}{\pi}+2(n^2 - 3).
    \end{equation*}
\end{theorem}

\begin{proof}
    The two-sided bound on the nullity follows from specializing (\ref{eq:twosidednullity}) to the case $g=0,$ since the proof of Theorem \ref{thm:S2lowerbound} shows $\sigma_-(f)$ and $h^1(\mathcal{G}_{n-1} \otimes \mathcal{G}_n)$ vanish when $g=0.$

     If $f$ and the $\Psi_1, \ldots, \Psi_{n-1}$ are unbranched then, as in the proof of Theorem \ref{thm:S2upperbound}, $\mathrm{R}(f) = 0$ and we get equality $\operatorname{Null}_\R(f) = \operatorname{N}(f)$ throughout (\ref{eq:twosidednullity}).
\end{proof}

\begin{cor}
    Let $f: S^2 \to S^{2n}$ be a linearly full, branched minimal immersion. Then
	\begin{equation}\label{eq:niceS2nulllower}
		\operatorname{Null}_\R f \geq (2n+3)(2n-2),
	\end{equation}
	with equality if and only if $f$ is twistor-equivalent to the Bor\r{u}vka immersion.
\end{cor}

\begin{proof}
    This follows from Theorem \ref{thm:S2NullityTheorem} using the same logic as the proof of Corollary \ref{cor:S2boruvlowerindex}.
\end{proof}

\subsection{The $g=1$ case}\label{ssect:genusone}

\indent \indent We next consider the case where $g=1,$ so $\Sigma \cong E$ is an elliptic curve. In this case, the canonical bundle $\mathcal{K}$ is trivial, so, specializing the discussion of \S\ref{ssect:degram}, we have
\begin{equation*}
    \mathcal{G}_r = \mathcal{O}(B_1 + \cdots + B_r), \quad a_r = \sum_{j=1}^r b_j,
\end{equation*}
and $0 \leq a_1 \leq \cdots \leq a_n$. In particular, 
\begin{equation*}
    \mathrm{R}(f) = \sum_{r=1}^{n-1} a_r.
\end{equation*}
The area identity (\ref{eq:areaformula}) implies $a_n > 0,$ so there is a lowest index $r$ with $a_r > 0.$ Denote this index by $\rho \geq 1.$

We have $\deg(\mathcal{G}_k \otimes \mathcal{G}_{k+1}) = a_k + a_{k+1},$ so, by Proposition \ref{prop:degreecohomobounds},
\begin{equation*}
    h^1(\mathcal{G}_k \otimes \mathcal{G}_{k+1}) = \begin{cases}
        1 & \text{if} \:\: k \leq \rho-2 \\
        0 & \text{if} \:\: k \geq \rho-1.
    \end{cases}
\end{equation*}
Therefore $\sigma_-(f) = \max \lbrace \rho-2, 0 \rbrace.$

\begin{lemma}\label{lem:g1asigbound}
    Suppose $f: E \to S^{2n}$ is a linearly full, totally isotropic branched immersion of an elliptic curve. For $n \geq 3,$
    \begin{equation*}
        a_n - \sigma_-(f) \geq  n + \left\lfloor\frac{n}{2} \right\rfloor + 3.
    \end{equation*}
\end{lemma}

\begin{proof}
    By Corollary \ref{cor:asumbound},
    \begin{equation*}
        a_n \geq 2n+1, \qquad \sum_{j=1}^n a_j \geq n(n+1)+2.
    \end{equation*}
    
    If $\rho \leq \lceil n / 2 \rceil,$ then $\sigma_-(f) \leq \lceil n/2\rceil-2$ since $n \geq 3.$ Therefore,
    \begin{equation*}
        a_n-\sigma_-(f) \geq 2n + 3 - \left\lceil \frac{n}{2} \right\rceil = n + \left\lfloor \frac{n}{2} \right\rfloor + 3.
    \end{equation*}
    
    On the other hand, if $\rho \geq \lceil n / 2 \rceil + 1,$ then $\sigma_-(f)=\rho-2$ and $2\rho\geq n+2$. We have that $a_j=0$ for $j<\rho$ and  $a_n \geq a_j$ for every $j,$ so
    \begin{equation*}
    \begin{aligned}
        (n-\rho+1)a_n &\geq \sum_{j=\rho}^n a_j = \sum_{j=1}^n a_j
        > n(n+1)\\
        &\geq n(n+1)+\rho(n+2-2\rho) =(n-\rho+1)(n+2\rho).
    \end{aligned}
    \end{equation*}
    This implies $a_n > n+2\rho,$ so
    \begin{equation*}
        a_n-\sigma_-(f) = a_n-\rho+2 > n+\rho + 2 \geq n + \left\lfloor\frac{n}{2} \right\rfloor + 3. \qedhere
    \end{equation*}
\end{proof}

\begin{lemma}\label{lem:g1cohomvanish}
    Suppose $f: E \to S^{2n}$ is a linearly full, totally isotropic branched immersion of an elliptic curve. Then
    \begin{equation*}
        H^1(\mathcal{N}) = 0, \quad H^1(\mathcal{E}_k) = 0, \:\: 1 \leq k \leq n-1.
    \end{equation*}
    Consequently, $\sigma_+(f) = 0.$
\end{lemma}

\begin{proof}
    Serre duality implies $h^1(\mathcal{E}_k) = h^0(\mathcal{E}_k^*).$ From the definition of $\mathcal{E}_k$ and $\mathcal{Q}_k,$ we get $\mathcal{E}_k^* = \mathcal{G}_k^* \otimes \mathcal{P}_{k+1}.$ The inclusion $\mathcal{P}_{k+1} \subset \underline{\C}^{2n+1}$ gives an inclusion $\mathcal{E}_k^* \subset (\mathcal{G}^*_{k})^{\oplus 2n+1}.$ If $a_k > 0,$ then $(\mathcal{G}^*_{k})^{\oplus 2n+1}$ has no holomorphic sections by Proposition \ref{prop:degreecohomobounds}. If $a_k = 0,$ then $\mathcal{E}_k^* = \mathcal{P}_{k+1} \subset \underline{\C}^{2n+1}.$ A holomorphic section of $\mathcal{P}_{k+1}$ corresponds to a constant vector $v \in \C^{2n+1}.$ Such a vector $v$ would be an element of $\mathcal{P}_{k+1}(p)$ for every $p \in E,$ and in particular $\langle v, f(p) \rangle_\C = 0$ for every $p \in E.$ Linear fullness of $f$ then implies $v = 0.$ The same argument applies to show $h^1(\mathcal{N}) = h^0(\mathcal{P}_2) = 0.$
\end{proof}

\begin{theorem}
    Suppose $f: E \to S^{2n}$ is a linearly full, totally isotropic branched immersion of an elliptic curve. Then
    \begin{equation*}
        (n-1)\frac{\operatorname{Area}(f)}{\pi} - 2\operatorname{R}(f) - 2\max \lbrace \rho-2, 0 \rbrace \leq \operatorname{Ind}_\R(f) \leq (n-1)\frac{\operatorname{Area}(f)}{\pi}.
    \end{equation*}
    and
    \begin{equation*}
        \frac{\operatorname{Area}(f)}{\pi} - 2b_1 \leq \operatorname{Null}_\R(f) \leq \frac{\operatorname{Area}(f)}{\pi} - 2b_1 + 4\operatorname{R}(f) + 4 \max \lbrace \rho-2, 0 \rbrace.
    \end{equation*}
    In particular, for $n \geq 3,$
    \begin{equation}\label{eq:g1lowbound}
        \begin{aligned}
            \operatorname{Ind}_\R(f) &\geq (n-2)\frac{\operatorname{Area}(f)}{\pi} + 2 \left( n + \left\lfloor \frac{n}{2} \right\rfloor + 3 \right) \\
            & \geq 2(n-1)(n^2+1) + 2 \left\lfloor \frac{n}{2} \right\rfloor.
        \end{aligned}
    \end{equation}
\end{theorem}

\begin{proof}
    The two-sided index and nullity bound follow from applying $\sigma_-(f) = \max \lbrace \rho-2, 0 \rbrace,$ $\sigma_+(f) = 0,$ and $h^1(\mathcal{G}_{n-1} \otimes \mathcal{G}_n) = 0$ to (\ref{eq:twosidedbound}) and (\ref{eq:twosidednullity}).

    For $n \geq 3,$ applying the result of Lemma \ref{lem:g1asigbound} to the general lower bound of Theorem \ref{thm:genlowbound} gives the first inequality of (\ref{eq:g1lowbound}). The second follows from the area bound (\ref{eq:generalarealowerbound}).
\end{proof}

\begin{rmk}
    In the case $n=2,$ if $f : E \to S^4$ is unbranched then the sum defining $\sigma_-(f)$ is empty and $\operatorname{R}(f) = b_1 = 0,$ so our calculations replicate the result of Montiel--Urbano \cite{MontUrb97},
    \begin{equation*}
        \operatorname{Ind}_\R(f) = \operatorname{Null}_\R(f) = \frac{\operatorname{Area}(f)}{\pi} \geq 16.
    \end{equation*}
\end{rmk}

\begin{rmk}
    For the case of higher genus $g>1$, we may also give an explicit lower bound in terms of $n,$ $g$, and the area alone. The ramification identities of \S\ref{ssect:degram} give
    \begin{equation*}
        \mathcal{G}_k \cong \mathcal{K}^{-k} \otimes \mathcal{O}(B_1+\cdots+B_k),
    \end{equation*}
    and Serre duality implies
    \begin{equation*}
        \begin{aligned}
        h^1(\mathcal{G}_k\otimes\mathcal{G}_{k+1}) & = h^0 \left( \mathcal{K}^{2k+2} \otimes  \mathcal{O} \left( -2\sum_{j=1}^k B_j - B_{k+1} \right) \right) \\
        &\leq h^0(\mathcal{K}^{2k+2}) = (4k+3)(g-1).
    \end{aligned}
    \end{equation*}
    This gives $\sigma_-(f)\leq(n-2)(2n+1)(g-1).$ Combining this with Theorem \ref{thm:genlowbound} and the estimate $a_n\geq2n+\min\{2n,g\}$ from Corollary \ref{cor:asumbound}, we obtain
    \begin{equation*}
        \operatorname{Ind}_{\R}(f) \geq (n-2)\frac{\operatorname{Area}(f)}{\pi} + 4n + 2 \min\{2n,g\} -(5n^2-7n-4)(g-1).
    \end{equation*}
    However, this estimate discards information about the ramification divisors and is generally weaker than the original bound. It seems likely to the authors that Brill--Noether theory could be used to improve several parts of this argument and give a stronger lower bound on the index depending only on $n,$ $g,$ and the area.
    \end{rmk}

\bibliographystyle{plain}
{\small
\bibliography{Refs}}
{\footnotesize
\Addresses}

\end{document}